\documentclass[11pt,a4paper,reqno]{amsart}
\usepackage{amsfonts,amssymb}
\usepackage{amsmath}
\usepackage{latexsym}
\usepackage[dvips]{graphicx}
\usepackage{color}
\usepackage{cite}
\usepackage{amsthm}
\usepackage{cancel}
\usepackage[normalem]{ulem}
\usepackage{soul}
\newcommand{\R}{\mathbb R}
\newcommand{\Z}{\mathbb Z}
\newcommand{\N}{\mathbb N}

\newcommand{\sgn}{\rm sign}

\newcommand{\Frac}[2]{\displaystyle \frac{#1 }{#2}}
\newcommand{\eps}{\varepsilon}
\newcommand{\toh}{\mathcal{T}[h_b]}
\newcommand{\mfT}{{\mathfrak T}}

\newtheorem{theorem}{Theorem} [section]
\newtheorem{lemma}{Lemma} [section]
\newtheorem{proposition}{Proposition} [section]

\newtheorem{definition}{Definition} [section]

\newtheorem{remark}{Remark}[section]
\newtheorem{hypothesis}{Hypothesis}

\let\ssection=\section\renewcommand{\section}{\setcounter{equation}{0}\ssection}

\newcommand{\tendsto}[1]{\renewcommand{\arraystretch}{0.5}
\begin{array}[t]{c}
\longrightarrow \\
{ \scriptstyle #1 }
\end{array}
\renewcommand{\arraystretch}{1}}

\title[Global solutions for the Boussinesq-Peregrine system]{Global solutions for the one-dimensional Boussinesq-Peregrine system under small bottom variation}
\author[Luc Molinet]{Luc Molinet$^1$}
\thanks{ $^1$Institut Denis Poisson, Universit\'e de Tours, Universit\'e d'Orl\'eans, CNRS, Parc Grandmont, 37200, France ({\tt  luc.molinet@univ-tours.fr}).}
\author[Raafat Talhouk]{Raafat Talhouk$^{2,3}$}
\thanks{ $^2$ Léonard de Vinci Pôle Universitaire, Research Center, 92 916 Paris La D\'efense, France. ({\tt raafat.talhouk@devinci.fr}).}
\thanks{ $^3$Laboratoire de math\'ematiques-EDST, Facult\'e des Sciences et EDST, Universit\'e Libanaise, Hadat, Liban. }

\begin{document}
\begin{abstract}
    
The Boussinesq-Peregrine system is derived from the water waves system in presence of topographic variation under the hypothesis of shallowness and small amplitude regime. The system becomes  significantly simpler (at least in the mathematical sens) under the hypothesis of small topographic variation.
In this work we study the long time and global well-posedness   of the Boussinesq-Peregrine system. We start by showing the intermediate time  well-posedness  in the case of  general topography (i.e. the amplitude of the bottom graph $\beta=O(1)$). The novelty resides in the functional setting,  $H^s(\R), \, s> \frac {1} {2}.$ Then  we show our main result establishing that  the global existence result obtained in \cite{MTZ} in the flat bottom case is  still valid for the Boussinesq-Peregrine system under the hypothesis of small amplitude bottom variation (i.e. $\beta =O(\mu)$). More precisely we prove that   this system is  unconditionally  globally well-posed in the Sobolev spaces of type $ H^s (\R), \, s> \frac {1} {2}$.  Finally, we  show  the existence of a weak global solution in the Schonbek sense \cite{Sch},  i.e.  existence of low regularity entropic solutions  of the small bottom amplitude Boussinesq-Pelegrine equations  emanating from $ u_0 \in H^1 $  and $ \zeta_0 $ in an Orlicz class   as weak limits of  regular solutions.  
\end{abstract}
\maketitle

\section{Introduction}
In this paper we are concerned with the  system introduced by D. H. Peregrine \cite{Per},  in the case of small amplitude bottom variation and one space dimension. The Peregrine system was derived  from the Euler water waves one  to describe weak amplitude $(\eps=O(\mu))$ long wave propagation at the surface of ideal incompressible liquid for irrotational flow  over variable bottom topography submitted to  gravitational force where the surface tension has been neglected.  This system is the generalization of the system introduced by J.  V.  Boussinesq in 1871 (\cite{Bou1}, \cite{Bou2}, see also \cite{BCS02}, \cite{BCS04})  to model the same flow over flat bottom.  
 The corresponding PDE's system,  is given by:

\begin{equation}\label{BP}
\left\{
\begin{array}{lcl}
\zeta_t+(hu)_x &=&0,\\
(1+\mu\toh)u_t +\zeta_x+\eps uu_x&=&0.
\end{array} \right.
\end{equation}

The quantity $h=1+\eps\zeta(x,t)-\beta b(x)$ corresponds to the normalized total height of the liquid and $1+\eps\zeta(x,t)$ describes the free surface of the liquid, $b$ being the function representing the bottom. $\eps$ and $\beta$ are related respectively to the free surface wave amplitude and the bottom one,  $x$ is the spatial position which is proportional to distance in the direction of propagation. The quantity $u(x,t)$ is the horizontal velocity field of the liquid particle which is at position $x$ at time $t.$ Finally the operator $\toh$ is given by: 
\begin{equation}\label{deftoh}
 \toh u=-\frac{1}{3h_b}\left( h_b^3 u_x \right)_x +\frac{\beta}{2h_b}\left[ \left( h_b^2b_x u \right)_x -h_b^2 b_x u_x \right]+\beta^2b_x^2u 
 \end{equation} 
 with $h_b=1-\beta b$.

Mésognon-Gireau in \cite{Mes} gave the first result concerning a long time existence  (i.e.  existence time of order 
$O(\frac{1}{\eps})$ for a modified Boussinesq-Peregrine equation (\ref{BP}) in
dimension 1.  The existence and uniqueness are showed in $H^N\times H_{\mu }^{N+1}, \, N\in \N), N\geq 4$. He also showed  a local  existence (existence time of order $O(1)$) and uniqueness result in all dimension.  V. Duch\^ene and S. Israwi  in \cite{DIs} have shown existence and uniqueness in  long time for the system (\ref{BP}) in  dimension one and two in a  regular functional sitting ($H^N\times H_{\mu}^{N+1},\, N\in \N), N\geq 4.$  These results are obtained for free amplitude bottom. (i.e.  no smallness hypothesis on the size of bottom amplitude $\beta$).  Here $ H_{\mu}^{s+1}$ stand for the Sobolev space  $ H^{s+1}(\R) $ endowed with the norm 
\begin{equation}\label{defHmu}
|u|_{H^{s+1}_\mu}^2 :=|u|_{H^s}^2+\mu|u_x|_{H^s}^2 \; .
\end{equation}
In this paper, we study the well-posedness of system \eqref{BP} in low Sobolev regularity $H^s\times H_{\mu }^{s+1},\; s>\frac{1}{2}$. In a first time, we show that the system (\ref{BP}) is  well-posed in intermediate time (i.e. existence time of order $O((\epsilon\vee \beta)^{-1}\mu^\frac{1}{2} )$. Then, we consider system  (\ref{BP}) under the hypothesis of small amplitude bottom $\beta=O(\mu)$  (see also \cite{Lanlivre13}, p. 155). This hypothesis enables to write 
\begin{equation}\label{toh}
\begin{array}{lll}
 1+\mu\toh u &=& 1-\frac{\mu}{3h_b}\left( h_b^3 u_x \right)_x +\frac{\mu\beta}{2h_b}\left[ \left( h_b^2b_x u \right)_x -h_b^2 b_x u_x \right]+\mu\beta^2b_x^2u \\ 
& &\\
 &=& {\displaystyle   1-\frac{\mu}{3}\,u_{xx}+O(\beta\mu)}.
 \nonumber
 \end{array}
 \end{equation} 

that leads to    the following simplified system after  dropping all terms of order  $O(\mu\beta)$ :

\begin{eqnarray}\label{BPW}
\left\{
\begin{array}{lcl}
\zeta_t+(hu)_x &=&0,\\
u_t +\zeta_x+\eps uu_x- \frac{\mu}{3}u_{txx}&=&0.
\end{array} \right.
\end{eqnarray}

Our main result is that this last system is unconditionally globally well-posedness  in $H^s \times H^{s+1}$ for $s> \frac{1}{2}$. As it was first noticing  by Schonbek (\cite{Sch})  for the "classical" Boussinesq equation, the crucial point consists in showing that the entropy functional associated with  the hyperbolic system, of which system \eqref{BPW}  may be seen as a perturbation, can be adapted to obtain a sort  of  "entropy" functional for \eqref{BPW}. It is worth emphasizing that, in contrast to the flat bottom case (see \cite{Sch} and  also \cite{MTZ}), in our  context we cannot prove that this functional is non increasing  but only that it is locally bounded, which is sufficient for our purpose.
Finally, we show  the existence of  weak global solution in the sense of Schonbek \cite{Sch} as a limit of regular solution.

In our knowledge these results for this system are new and pertinent. In particular this   is the first global existence result in the context of non flat bottom for Boussinesq type system.

 \subsection{Statement of the main results}
 \begin{definition}\label{def}  Let $ s>1/2,\,T>0$ and $b\in H^{s+1}\cap W^{\frac{5}{2}^+,\infty}$. We will say that $(\zeta,u)\in L^\infty(]0,T[;H^s\times H^{s+1}_\mu) $ is a solution to \eqref{BP} (or \eqref{BPW}) associated with the initial datum $ (\zeta_0,u_0) \in H^s(\R)\times H^{s+1}_\mu(\R) $  if
  $ (\zeta,u) $ satisfies \eqref{BP} (or \eqref{BPW})   in the distributional sense, i.e. for any test function $ \psi\in C_c^\infty(]-T,T[\times \R) $,  it holds
  \begin{equation}\label{weak}
  \left\{ 
 \begin{array}{l}

{ \displaystyle \int_0^\infty \int_{\R} \Bigl[(\psi_t  +\psi_x u (1+\eps\zeta-\beta b ) \Bigr] \, dx \, dt +\int_{\R} \psi(0,\cdot) \zeta_0 \, dx =0}\\
 { \displaystyle  \int_0^\infty \int_{\R} \Bigl[(1+\mu\toh)\psi_t u +\psi_x\zeta +  \psi_x  u^2/2 \Bigr] \, dx \, dt +\int_{\R} \psi(0,\cdot) u_0 \, dx =0}
  \end{array}
  \right.
  \end{equation} 
 \end{definition}
 
 \begin{remark} \label{rem2} Note that $ H^s(\R) $ is an algebra for $ s>1/2 $ and thus  $ \zeta u $ and $ u^2 $ are well-defined and belong to $L^\infty(]0,T[\, ;H^s(\R) $.   
  Therefore \eqref{weak} forces $ (\zeta_t,u_t)  \in L^\infty(]0,T[ \, ; H^{s-2}(\R)\times H^{s+1}) $  and  thus \eqref{BP} (or  \eqref{BPW}) is satisfied in $L^\infty(]0,T[ \,; H^{s-2}(\R)\times H^{s+1}) $. In particular, $ (\zeta,u)\in C([0,T] \, ; H^{s-2}(\R)\times H^{s+1})$ and \eqref{weak} forces 
 $ (\zeta(0),u(0))=(\zeta_0,u_0)$.
  \end{remark}
 \begin{definition}\label{def2}  Let  $ s>1/2$ and $ b\in H^{s+1}(\R)$. We will say that the Cauchy problem associated with \eqref{BPW} is unconditionally globally well-posed in $ H^s(\R )\times H^{s+1}_\mu(\R) $ if 
 \begin{enumerate}
 \item for any initial data $ (\zeta_0,u_0)\in H^s( {\R})\times H^{s+1}_\mu(\R)  $, satisfying  $1+\varepsilon \zeta_0-\beta b >0 $ on $\R$,  there exists a solution
 $ (\zeta,u) \in C(\R_+\, ; H^s(\R)\times H^{s+1}_\mu(\R) ) $ to \eqref{BPW} emanating from $ (\zeta_0,u_0) $.
 \item for any $ T>0$, this solution   is the unique solution to  \eqref{BPW} associated with $ (\zeta_0,u_0) $ that belongs to  $ L^\infty(]0,T[\, ;H^s(\R)\times H^{s+1}_\mu(\R) )$.
 \item  for any $ T>0 $, the solution-map
  $(\zeta_0, u_0) \mapsto (\zeta,u) $ is continuous from  $ H^s(\R)\times H^{s+1}_\mu(\R) $    into $C([0,T]\,; H^s(\R)\times H^{s+1}_\mu(\R)) $.
  \end{enumerate}
  
 \end{definition}

 \begin{theorem}\label{maintheorem}
 For any $ s>1/2 $ and  $b\in H^{s+1}(\R),$ then the Cauchy problem (\ref{BPW})  is unconditionally globally well-posed 
in $H^s(\R)\times H^{s+1}_\mu(\R)$ in the sense of Definition \ref{def2} .
\end{theorem}
\begin{theorem}\label{T7}
Let  $(\zeta_0,u_0)\in  \Lambda_{\sigma_0}\times H^1_\mu$ { (in particular $1+\eps\zeta_0- \beta b >0$ {a.e. on $ \R $)}} then there exists a weak solution for the Boussinesq system (\ref{BPW}) with initial data $(\zeta_0,u_0)$ verifying that $(\zeta,u) \in L^\infty_{loc}(\R^+; \Lambda_{\sigma_0}\times H^1_\mu).$
\end{theorem}
 For the definition of $\Lambda_{\sigma_0}$ see (\ref{Orlicz}).\\
For the system \eqref{BP} we also derived an intermediate  time existence result under the following hypothesis on the parameter $ \beta $. 
\begin{hypothesis}\label{hyp2}
$$
0<\beta\le C_0=C_0(\frac{1}{h_0},\|b\|_{W^{\frac{5}{2}+,\infty}})
$$
where $ C_0 $ is the constant that appears in Lemma \ref{proprim'}.
\end{hypothesis}

This condition ensures that the algebraic inverse of the operator  $$\mfT_b=h_b(1 +\mu \toh)$$  is Lipschitz from $H^\theta(\R)$ to $ H^{\theta+1}_\mu(\R) $ for  $ -3/2<\theta\le 0$.

\begin{theorem}\label{theo3}
Let $ \mu>0 $, $ \varepsilon>0 $, $ \beta>0 $, \; $s>1/2 $, $ b\in  W^{s+2,\infty}(\R) $. Assume that  Hypothesis \ref{hyp2} is  satisfied and that  there exists 
 $ h_0 >0 $ such that 
\begin{equation}\label{hb}
 h_b:=1-\beta b \ge h_0 >0 \qquad {\rm on  } \; \R \; .
\end{equation}
Then  the Cauchy problem associated with \eqref{BP} is unconditionally locally well-posed in $H^s(\R) \times H^{s+1}_\mu(\R) $. 

Moreover, the maximal time of exitence $ T^*$ of the solution satisfies 
$$
T^* \gtrsim  T_0\sim (\epsilon\vee \beta)^{-1} \Bigl(1+\mu^{-1/2}  ( |\zeta_0|_{H^{\frac{1}{2}+}}+
|u_0|_{H^{\frac{3}{2}+}_\mu})\Bigr)^{-1}
$$
 and  
 $$
  \sup_{t\in [0,T_0]}| \zeta(t)|_{H^s}^2 + 
  | u(t)|_{H^{s+1}_\mu}^2  \lesssim 
  |\zeta_0|_{H^s}^2+
|u_0|_{H^{s+1}_\mu}^2
 $$
\end{theorem}

\begin{remark}
We do not know if we can reach the long time existence (existence time of order $0(\frac{1}{\eps\vee \beta}) $ for \eqref{BP}. 
The loss of the factor $\mu^{1/2} $ with respect to this long time existence result seems to follow from a loss of symmetry in the formulation of \eqref{BP} with respect to the 1D Green-Naghdi system. Indeed to inverse the operator $1+\mu\toh u $ we need to multiply \eqref{BP}$_2$ by $ h_b$ whereas a multiplication by $ h $ would be needed to obtain a cancellation with $ (h u)_x $ that appears in \eqref{BP}$_1$. However this difficulty can be dropped if we consider an approximate consistent version of system (\ref{BP})  up to $O(\epsilon\mu)$ (which is of order $O(\mu^2))$ by  multiplying the second equation of (\ref{BP}) by $h$ and writing 
$$ \mu h\toh u= \mu h_b \toh u + O(\epsilon\mu).$$ 
System (\ref{BP}) becomes 
\begin{equation}\label{BPModified}
\left\{
\begin{array}{lcl}
\zeta_t+(hu)_x &=&0,\\
(h+\mu h_b\toh)u_t +h\zeta_x+\eps h uu_x&=&0.
\end{array} \right.
\nonumber
\end{equation}
For this system theorem \ref{3}, at least in a regular functional setting (say $s>\frac{3}{2}$), still valid with long time existence, i.e. with time $T^*$ such that 
$$ T^* \gtrsim  T_0\sim (\epsilon\vee \beta)^{-1}   \Bigl( |\zeta_0|_{H^{\frac{3}{2}+}}+
|u_0|_{H^{\frac{5}{2}+}_\mu}\Bigr)^{-1}\; .$$
\end{remark}

 \begin{remark}
     In the case of weak bottom variation hypothesis the results of theorem \ref{3} still valid for system \eqref{BPW}  without Hypothesis \ref{hyp2} and  in a longer time $$ T^* \gtrsim  T_0\sim (\epsilon\vee \beta)^{-1}   \Bigl( |\zeta_0|_{H^{\frac{1}{2}+}}+
|u_0|_{H^{\frac{3}{2}+}_\mu}\Bigr)^{-1}\; . $$
However, one has to replace \eqref{hb} by a similar hypothesis on $ h$ at time $0 $, i.e. $h(0,\cdot)\ge h_0>0 $ on $\R$. Indeed, under this hypothesis, one can multiply \eqref{BPW}$_2$ by $ h $ to obtain the desired cancellation with the term $ (h u)_x $ that appears in \eqref{BPW}$_1$.
\end{remark}
 The rest of the paper is organized as follows:  The second section is dedicated to the notations and preliminary results that will be useful throughout the paper. In the third section we prove the local well-posedness of \eqref{BP} in $ H^s\times H^{s+1}_\mu $, $s>1/2 $, under a smallness assumption on $ \beta $. 
In Section 4, we explain how a local existence result for \eqref{BPW} can be easily obtained by following the same steps as in Section 3 with obvious simplifications and we show that the local solution constructed   satisfies an  entropic energy estimate. Using this estimate and  the results established previously, we  prove our main theorem \ref{maintheorem}.  Finally in Section  5 we recover the result of M. E. Schonbek concerning the existence of a weak entropic solution for the Boussinesq system and this using a sequence of regular solutions given by our main theorem. 

 \section{Notations and preliminary}
 \subsection{Notations and function spaces.}
 In the following, $C$ denotes any non negative constant whose exact expression is of no importance. The notation $a\lesssim b$ means that $a\leq C_0 b$.\\
 We denote by $C(\lambda_1, \lambda_2,\dots)$ a non negative constant depending on the parameters
 $\lambda_1$, $\lambda_2$,\dots and whose dependence on the $\lambda_j$ is always assumed to be nondecreasing.\\
 Let $p$ be any constant with $1\leq p< \infty$ and denote $L^p=L^p(\R)$ the space of all Lebesgue-measurable functions
 $f$ with the standard norm $$\vert f \vert_{L^p}=\big(\int_{\R}\vert f(x)\vert^p dx\big)^{1/p}<\infty.$$ When $p=2$,
 we denote the norm $\vert\cdot\vert_{L^2}$ simply by $\vert\cdot\vert_2$. The real inner product of any functions $f_1$
 and $f_2$ in the Hilbert space $L^2(\R)$ is denoted by
\[(f_1,f_2)=\int_{\R}f_1(x)f_2(x) dx.
 \]
 The space $L^\infty=L^\infty(\R)$ consists of all essentially bounded, Lebesgue-measurable functions
 $f$ with the norm
\[
 \vert f\vert_{\infty}= \hbox{ess}\sup \vert f(x)\vert<\infty.
\]
For convenience, we denote the norm of $L^\infty(\R_+^*\times \R)$ by $\|\cdot\|_{L^\infty_{t,x}}.$\\
For any real constant $s\geq0$, $H^s=H^s(\R)$ denotes the Sobolev space of all tempered
 distributions $f$ with the norm $\vert f\vert_{H^s}=\vert \Lambda^s f\vert_2 < \infty$, where $\Lambda$
 is the pseudo-differential operator $\Lambda=(1-\partial_x^2)^{1/2}$.\\
\begin{definition}
For all $ s\in\R $, $ \mu>0 $ and $ f\in H^{s+1}(\R) $ we set 
$$
|f|_{H^{s+1}_\mu} ^2:=|f|_{H^s}^2+\mu |f_x |_{H^s}^2 \; .
$$ 
\end{definition}
For $ s>0 $,  we denote by $W^{s,\infty}=W^{s,\infty}(\R)=\{f, \Lambda^s f\in L^{\infty}\}$ endowed with its canonical norm. \\
 For any functions $u=u(t,x)$ and $v(t,x)$ 
 defined on $ [0,T)\times \R$ with $T>0$, we denote the inner product, the $L^p$-norm and especially
 the $L^2$-norm, as well as the Sobolev norm,
 with respect to the spatial variable $x$, by $(u,v)=(u(t,\cdot),v(t,\cdot))$, $\vert u \vert_{L^p}=\vert u(t,\cdot)\vert_{L^p}$,
 $\vert u \vert_{L^2}=\vert u(t,\cdot)\vert_{L^2}$ , and $ \vert u \vert_{H^s}=\vert u(t,\cdot)\vert_{H^s}$, respectively.\\
For $(X,\|\cdot\|_X)$ a Banach space, we denote as usually $L^p(]0,T[;X)$, $1\leq p\leq +\infty $,  the space of mesurable functions equipped by the norm:
 \[\big\Vert u\big\Vert_{L^p_TX} \ =\left( \displaystyle \int_0^T\big\Vert u(t,\cdot)\big\Vert_{X}^p \right) ^{1/p}  \hbox{ for }1\leq p<+\infty,\]
and
 \[\big\Vert u\big\Vert_{L^\infty_T X} \ = \ \hbox{ess}\sup_{t\in]0,T[}\Vert u(t,\cdot)\Vert_{X}\;\; \hbox{ for } p=+\infty  \ .\]
Finally, $C^k([0,T]; X)$  is the space of  $k$-times continuously differentiable functions from $[0,T]$ with value in $X,$ equipped with its standard norm \[\big\Vert u\big\Vert_{C^k([0,T];X)} \ = \ \max_{0\leq l\leq k} \sup_{t\in[0,T]}\vert u^{(l)}(t,\cdot)\vert_{X} \ .\] 
 Let $C^k(\R)$ denote the space of
 $k$-times continuously differentiable functions.  \\
 For any closed operator $T$ defined on a Banach space $X$ of functions, the commutator $[T,f]$ is defined
  by $[T,f]g=T(fg)-fT(g)$ with $f$, $g$ and $fg$ belonging to the domain of $T$. 
  
Throughout the paper, we fix a smooth even bump function $\eta$ such that
\begin{equation}\label{defeta}
\eta \in C_0^{\infty}(\mathbb R), \quad 0 \le \eta \le 1, \quad
\eta_{|_{[-1,1]}}=1 \quad \mbox{and} \quad  \mbox{supp}(\eta)
\subset [-2,2].
\end{equation}
We set  $ \phi(\xi):=\eta(\xi)-\eta(2\xi) $. For $l \in \mathbb \N\setminus\{0\}$, we define
\begin{displaymath}
\phi_{2^l}(\xi):=\phi(2^{-l}\xi). 
\end{displaymath}
By convention, we also denote
\begin{displaymath}
\phi_{1}(\xi):=\eta(\xi).
\end{displaymath}
Any summations over capitalized variables such as $N, $, $K$  are presumed to be dyadic. {Unless stated otherwise, we work with non-homogeneous decompositions for space, time and  modulation variables, i.e. these variables range over numbers of the form  $\{2^k : k\in\mathbb N\}$ respectively.}  Then, we have that
{
\begin{displaymath}
\sum_{N\ge 1}\phi_N(\xi)=1\quad \forall \xi\in \R, \quad \mbox{supp} \, (\phi_N) \subset
\{\frac{N}{2}\le |\xi| \le 2N\}, \ N \in \{2^k : k\in \N\setminus\{0\}\},
\end{displaymath}
}
Let us now define the following Littlewood-Paley multipliers : 
$$
P_Nu:=\mathcal{F}^{-1}_x\big(\phi_N\mathcal{F}_xu\big), \quad \widetilde{P}_N:= \hspace*{-5mm}\sum_{ N/4 \le K \le 4N} \hspace*{-5mm} P_K,
  \quad P_{\gtrsim  N}:=\sum_{ K \gtrsim  N}  \quad \text{and} \quad   \quad P_{\ll N} :=\sum_{1\le K \ll N} P_{K}, 
 $$
\subsection{Some preliminary estimates.} The following product and commutator estimates will be used intensively throughout the paper. 
\begin{proposition}\label{propcom} Let  $ N\gg 1 $ then 
\begin{equation}\label{cm1} |[P_N,P_{\ll N} f]g_x|_{L^2}  \lesssim |f_x|_{L^\infty}  |\tilde{P}_N g|_{L^2},
\end{equation}
\end{proposition}
We give a short proof of \eqref{cm1} in the appendix for the sake of completeness.

We will also need the two following product estimates in Sobolev spaces :
\begin{enumerate}
\item 
  For every $  p,r,t $ such that $ r+p>t+1/2>0 $ and $r,p\ge t  $,
\begin{equation}
|fg|_{H^t} \lesssim |f |_{H^p} |g|_{H^r}
\; . \label{prod2}
\end{equation}
\item For any $ s\ge 0 $ 
\begin{equation}
|fg|_{H^s} \lesssim |f |_{L^\infty} |g|_{H^s}+  |f|_{H^s}|g |_{L^\infty}
\; . \label{prod3}
\end{equation}
\end{enumerate}
Inequality (\ref{prod2}) is a standart Sobolev product estimate, the second one (\ref{prod3}) is the well known Moser product estimate (see for instance \cite{Taylor91} or  \cite{Lanlivre13}, and references therein.)
With \eqref{prod2}-\eqref{prod3}  in hand, it is straightforward (see Appendix) to prove the two following  frequency localized product estimates  given in Proposition (\ref{propproduit}) that we will extensively use in the next section.
\begin{proposition}\label{propproduit}
For any $ N\gg1  $ and $ s>0 $  it holds 
\begin{equation}
N^s |P_N ( P_{\gtrsim N} f \, g_x)|_{L^2} \lesssim \delta_N \min \Bigl( |f|_{H^{s+1}} |g|_{L^\infty}, |f|_{H^s} |g_x|_{L^\infty}\Bigr) \label{proN1}
\end{equation}
 whereas for $ s>1/2 $ it holds 
\begin{equation}
N^{s-1} |P_N ( P_{\gtrsim N} f g_x)|_{L^2} \lesssim \delta_N |f|_{H^{s+1}} |g|_{H^{s-1}}\label{proN2}
\end{equation}
with $ |(\delta_{2^j})_{j\in \N}|_{l^2}\le 1 $. 
\end{proposition}
We will also need the following product and commutator estimates to control the multiplication by a smooth bounded function (see the proof in the Appendix).
\begin{proposition}\label{propproduit2}
 Let $ \theta\in \R $, $ g\in H^\theta $ and $ f \in W^{|\theta|+1+,\infty} $. Then it holds  
\begin{equation}
|fg|_{H^\theta} \lesssim |f |_{W^{|\theta|+,\infty}} |g |_{H^\theta}
\;  \label{prod5}
\end{equation}
 and for any $ N\ge 1$ 
\begin{equation}
N^\theta  |[P_N, f ] g_x |_{L^2} \lesssim   \delta_N |f_x |_{W^{|\theta|+,\infty}} |g |_{H^\theta}
\; . \label{prod6}
\end{equation}
with $ |(\delta_{2^j})_{j\in \N}|_{l^2}\le 1 $. 
\end{proposition}
\section{Local existence  and energy estimates for \eqref{BP}}\label{section3}
In all this section we assume that $ b\in W^{s+2}(\R) $. This is not necessary and actually to prove the  LWP of \eqref{BP} in $ H^s(\R)\times H^{s+1}(\R) $, $ b\in C^{[s]+3}_b(\R) $ would be  enough.
\subsection{Local well-posedness and estimates for a Bona-Smith's approximation}\label{31}
We fix $ \epsilon, \mu, \beta >0 $ in (\ref{BP}). For $ \nu>0  $ we 
 consider the  Bona-Smith type regularization  problem associated to (\ref{BP})
\begin{eqnarray}\label{3l}
\left\{
\begin{array}{rcl}
\zeta_t- \nu \zeta_{txx} + (h_b u)_x +\epsilon (\zeta u)_x&=&0,\\
(1+\mu \toh) u_t+\zeta_x + \epsilon u u_x &=&0,\\
(\zeta,u)(0)&=&(\zeta_0,u_0)\; . 
\end{array} \right.
\end{eqnarray}
Since $ h_b>0 $ on $\R $, multiplying the second equation in \eqref{3l} by $ h_b $ and  setting $ V=(\zeta,u)  $, \eqref{3l} can be rewritten as 
\begin{equation}
\frac{d}{dt} V = \Omega_\nu (V) 
\end{equation}
where 
$$
 \Omega_\nu (V)=\Bigl( (1-\nu \partial_x^2)^{-1}[- ((1-\beta b) u)_x  -\epsilon (u \zeta)_x ],\;
\mfT^{-1}_b[ -h_b \partial_x (\zeta +\frac{\epsilon}{2}  u^2) ]\Bigr) 
$$
with $ \mfT_b=h_b(1+\mu \toh$ (see \eqref{deftoh}).
Recall that $ H^s(\R) $ is an algebra for $s>1/2$. 
Since according to \eqref{estT} in Lemma \ref{proprim'}, for $ s>1/2$, $ \mfT^{-1}_b $ is continuous from $ H^{s} (\R) $ to $ H^{s+1}(\R)$ and this is also the case for $ (1-\nu\partial_x^2)^{-1}$, it is straightforward to check that $ \Omega_\nu $ is a  locally Lipschitz  mapping from $ (H^{s+1}(\R))^2 $ into itself for $ s>1/2$. Therefore by the Cauchy-Lipschitz theorem  in Banach spaces we infer that 
\eqref{3l} is locally well-posed in $(H^{s+1}(\R))^2 $, i.e. for any $ (\zeta_0,u_0)\in  (H^{s+1}(\R))^2 $ there exists $ 
 T_{s,\mu,\nu}=T_{s,\mu,\nu}(|\zeta_0|_{H^{s+1}}+|u_0|_{H^{s+1}})$ and a unique solution $(\zeta,u) \in C^1([0,T_{s,\mu,\nu}]; (H^{s+1})^2)$. Moreover, for any $R>0 $,  the mapping that to $ (\zeta_0,u_0) $ associates $(\zeta,u) $ is continuous from $ ( B(0,R)_{H^{s+1}})^2\subset (H^{s+1})^2 $ into 
 $ C([0,T_{s,\mu,\nu}(R)]; (H^{s+1})^2)$.

We start  by stating some energy estimates fundamental to prove our result.
For $ s\ge 0 $ and $ \nu\ge 0  $ we define $ E^s_{\mu,\nu} \; :\; (H^{s+1}(\R))^2 \to \R $ by 
\begin{equation}\label{ener}
E^s_{\mu,\nu}(\zeta,u) =|\zeta|_{H^s}^2+\nu |\zeta_x|_{H^s}^2+|u|_{H^{s+1}_\mu}^2
\end{equation}
In the sequel we denote by $ (\delta_N)_{N\in 2^{\Z}} $  a generic  sequence of positive real numbers such that 
$$ \sum_{j\in \N} \delta_{2^j}^2 \le 1\ .$$ 

\subsubsection{$H^s$ estimate.} 
Applying the operator $P_N$, $ N\ge 1 $,   to the first equation in (\ref{3l}),  multiplying by $
N^{2s}P_N \zeta$ and   integrating with respect to $x$ the resulting equation, we get
\begin{align}
\displaystyle \frac{N^{2s} }{2}  \frac{d}{dt} (|P_N\zeta|_{L^2}^2+ \nu |P_N \zeta_x|^2_{L^2}) &= -\epsilon N^{2s} (P_N(\zeta u)_x ,P_N \zeta)_{L^2}  \nonumber\\
& \qquad -N^{2s} ( P_N( h_b u)_x ,P_N \zeta)_{L^2}  \; ,\label{E1s}
\end{align}  
where, by  integrating by parts, it holds  
 
 \begin{align}
  -(P_N(h_b u)_x,P_N  \zeta)_{L^2}& =  (P_N(h_b u),P_N  \zeta_x)_{L^2}  \nonumber\\
 &=(h_b P_N u, P_N \zeta_x)_{L^2} +([P_N, h_b] u , P_N \zeta_x)_{L^2} \nonumber \\
 &= (h_b P_N u, P_N \zeta_x)_{L^2}  - (\partial_x([P_N, h_b] u ), P_N \zeta)_{L^2}   \nonumber\\
 &= (h_b P_N u, P_N \zeta_x)_{L^2}  - ([P_N, \partial_x h_b] u , P_N \zeta)_{L^2}  \nonumber \\
 & \qquad -([P_N, h_b] u_x , P_N \zeta)_{L^2} \; .
 \label{coml}
 \end{align}
 Now we notice that 
\begin{equation}\label{selfadjoint}
\mfT_b v = -\frac{\mu}{3} (h_b^3 v_x)_x + \Bigl( h_b + \frac{\beta\mu}{2}\partial_x (h_b^2 b_x)+\beta^2 \mu  \, b_x \Bigr) v
\end{equation}
so that $  \mfT_b $ is  $ L^2_x $ self-adjoint and it holds 
\begin{equation}\label{selfPN}
P_N (\mfT_b v) =  \mfT_b P_N v - \frac{\mu}{3} \partial_x\Bigl([P_N, h_b^3 \partial_x] v\Bigr) + 
 [P_N, g_b] v
\end{equation}
where 
\begin{equation}\label{defgbeta}
g_b:= h_b + \frac{\beta\mu }{2}\partial_x (h_b^2 b_x)+\beta^2 \mu\,  b_x^2
\end{equation}
belongs to $ C^1_b(\R) $ (as soon as $b\in C^3_b(\R) $) and satisfies $ |g_b|_{L^\infty} \lesssim 1 $ and $ |\partial_x g_b|_{L^\infty_x} \lesssim \beta $.
 
Multiplying  the second  equation in (\ref{3l}) by $ h_b$ we get 
\begin{equation}\label{eqT}
\mfT_b u_t=- h_b \zeta_x - \epsilon h_b u u_x =-\partial_x(h_b \zeta+\frac{\epsilon}{2} h_b u^2) + \partial_x h_b \zeta + \frac{\epsilon}{2} \partial_x h_b u^2 \; .
\end{equation}
 According to \eqref{eqT} and Lemma \ref{proprim'} we thus  infer that for $ s>1/2 $, 
\begin{align}
| u_t|_{H^{s}_\mu} & \lesssim  |\partial_x(h_b \zeta+\frac{\epsilon}{2} h_b u^2)|_{H^{s-1}} + |\partial_x h_b \zeta + \frac{\epsilon}{2} \partial_x h_b u^2|_{H^{s-1}} \nonumber \\
&\lesssim  
 C(|b|_{W^{s+1+,\infty}})\, \Bigl(  |\zeta|_{H^s}+ \epsilon |u|_{L^\infty_x} |u|_{H^s} \Bigr)\; .\label{estut}
   \end{align}
Note that for $ 1/2<s<1 $ we  have $ s-1<0 $ and thus we already have to assume Hypothesis \ref{hyp2} here.
Applying the operator $ P_N $ to \eqref{eqT}, multiplying by $
N^{2s}P_N u$,  integrating with respect to $x$ the resulting equation and using \eqref{selfPN}, we get
\begin{align}
\displaystyle \frac{N^{2s} }{2}  \frac{d}{dt} (\mfT_b P_N u, P_N u)_{L^2_x} &= -N^{2s}  (h_b P_N \zeta_x, P_N u)_{L^2} -\epsilon N^{2s}  ( P_N ( h_b u u_x)  ,P_N u)_{L^2}\nonumber \\
& -N^{2s} \Bigl([P_N, h_b] \zeta_x - \frac{\mu}{3} \partial_x([P_N, h_b^3 \partial_x] u_t) \bigr. \nonumber\\ 
&\Bigl. \hspace{5cm}+ [P_N, g_b] u_t   ,P_N u\Bigr)_{L^2}.\label{E2s}
\end{align} 
Summing the identities \eqref{E1s} and \eqref{E2s} and noticing that the contribution of the first term in the right-hand side of \eqref{coml} vanishes with the contribution of the first term in the right-hand side of \eqref{E2s} we obtain 

\begin{align}
\displaystyle \frac{N^{2s} }{2}  \frac{d}{dt} \Bigl(|P_N \zeta|_{L^2}^2+ &\nu |P_N \zeta_x|^2_{L^2}+ (\mfT_b P_N u, P_N u)_{L^2_x} \Bigr)
=-\epsilon N^{2s}\Bigl(   P_N ( h_b u u_x) ,P_N u\Bigr)_{L^2}\nonumber \\
& -\epsilon N^{2s} (P_N(\zeta u)_x ,P_N \zeta)_{L^2} -N^{2s} ([P_N, \partial_x h_b] u , P_N \zeta)_{L^2}\nonumber \\
&  -N^{2s} ([P_N,  h_b] u_x , P_N \zeta)_{L^2}\nonumber \\
& - N^{2s} \Bigl([P_N, h_b ] \zeta_x - \frac{\mu}{3} \partial_x([P_N, h_b^3 \partial_x] u_t) + 
[P_N, g_b] u_t   ,P_N u\Bigr)_{L^2}
.\label{E3s}
\end{align} 
Let us estimate one by one the terms appearing in the right-hand side of \eqref{E3s}. 
 For $N\lesssim 1 $, \eqref{prod5} directly yields
 $$
N^{2s} |(P_N(h_b u u_x),P_N  u)_{L^2}|\lesssim |h_b (u^2)_x|_{H^{-1}} |u|_{L^2_x} \lesssim 
 C(|b|_{W^{1+,\infty}})\, |u|_{L^\infty_x} |u|_{L^2_x}^2 \; .
 $$
 Now for $ N\gg 1$, we first decompose this term as follows
$$
P_N(h_b u u_x) = P_N (P_{\ll N} (h_b u) \, u_x) + P_N (P_{\gtrsim N} (h_b u) u_x)= A_{1,N}+A_{2,N} \; 
$$
where on account of \eqref{cm1} and using integration by parts we get 
\begin{align*}
N^{2s} |(A_{1,N},P_N u)_{L^2_x}| & \lesssim \delta_N  |(\partial_x(h_b u) |_{L^\infty_x} |u|_{H^s}^2 \\
&\lesssim C(|b|_{W^{1,\infty}})\, \delta_N |u|_{H^s}^2   \Bigl(\beta  |u|_{L^\infty_x} +|u_x|_{L^\infty_x} \Bigr)   \;.
\end{align*}
On the other hand, we may rewrite $ A_{2,N} $ as 
$$
A_{2,N}=  P_N \Bigl( \sum_{N_1\gtrsim N} P_{N_1} (h_b u) P_{\lesssim N_1} u_x\Bigr)
$$
that leads, by using \eqref{prod5},  to 
\begin{align*}
N^s |A_{2,N}|_{L^2_x} &\lesssim \sum_{N_1\gtrsim N}
\Bigl(\frac{N}{N_1}\Bigr)^s \delta_{N_1}  |h_b u |_{H^s} |u_x |_{L^\infty_x}\\
&  \lesssim \delta_N |h_b u |_{H^s} |u_x |_{L^\infty_x} \\
& \lesssim C(|b|_{W^{s+,\infty}})\, \delta_N  |u|_{H^s} |u_x |_{L^\infty_x} \; .
\end{align*}
Therefore we eventually get for $ N\ge 1$, 
$$
\epsilon N^{2s}|(P_N(h_b u u_x),P_N  u)_{L^2}|\lesssim C(|b|_{W^{s+1,\infty}})\, \epsilon\, \delta_N^2 (|u|_{L^\infty_x}+|u_x|_{L^\infty_x}) |u|_{H^s}^2 \; .
$$
At this point we notice that Sobolev inequality leads for $s\in ]1/2,3/2] $ to  
$$
|u_x|_{L^\infty} \lesssim |u|_{H^s}^{(s-1/2)-}   |u|_{H^{s+1}}^{(3/2-s)+}\lesssim \mu^{\frac{1}{2}(s-3/2)-} |u|_{H^{s+1}_\mu}
$$
whereas $|u_x|_{L^\infty} \lesssim |u|_{H^s_\mu} $ for $ s>3/2 $. This means that in the estimate of this term we loose a factor 
$\mu^{\frac{1}{2}(s-3/2)-}\lesssim \mu^{-1/2} $ when working at  regularity $ s\in ]1/2,3/2] $. Even if this factor do depend on $ s>1/2$ (it disappears for $ s>3/2$) , since we will loose anyway a factor $ \mu^{-\frac{1}{2}} $ in the estimate of the next term we do not take this into consideration  and simply keep the baddest case (i.e. $ s=\frac{1}{2}+ $)  to get 
\begin{equation}
\epsilon N^{2s}|(P_N(h_b u u_x),P_N  u)_{L^2}|\lesssim \epsilon\, \delta_N^2\mu^{-\frac{1}{2}} |u|_{H^{s+1}_\mu}  |u|_{H^s}^2 \; .\label{tod1}
\end{equation}
Now, we tackle the estimate of the second term of the right-hand side of \eqref{E3s}.  At this stage, it is worth noticing that there is a kind of lost of symmetry in the Boussinesq-Peregrin system since $ h$ appears in the first equation whereas only $ h_b$ appears in the second equation. Somehow, we will pay this lost of symmetry here since we will  loose a factor $\mu^{-1/2} $ even in high regularity. Indeed by using  \eqref{prod3} we get, for $ N\gg 1 $, 
\begin{align}
\epsilon\, N^{2s} |(P_N(u_x \zeta),P_N  \zeta)_{L^2}|& \lesssim \epsilon\,N^{s} \delta_N |u_x \zeta|_{H^s}  |P_N \zeta|_{L^2} \nonumber\\
 & \lesssim \epsilon\ \delta_N^2 \Bigl( |u|_{H^{s+1}} |\zeta|_{L^\infty} +|u_x|_{L^\infty} |\zeta|_{H^s} \Bigr) | \zeta|_{H^s} \nonumber\\
 & \lesssim  \epsilon\ \delta_N^2 \Bigl( \mu^{-1/2}  |\zeta|_{L^\infty} +|u_x|_{L^\infty} \Bigr) | \zeta|_{H^s}(|\zeta|_{H^s}+|u|_{H^{s+1}_\mu})\label{lost}
\end{align}
On the other hand, integrating by parts and using \eqref{cm1}  we obtain
\begin{align*}
\epsilon\, N^{2s}|(P_N(u \zeta_x),P_N  \zeta)_{L^2}|&=\epsilon\,N^{2s} \Bigl|-\frac{1}{2} (P_{\ll N} u_x P_N \zeta, P_N  \zeta)_{L^2} \\ 
&+([P_N, P_{\ll N} u]  \zeta_x ,P_N  \zeta)_{L^2} 
 + \Bigl(P_N(P_{\gtrsim N} u \,  \zeta_x), P_N  \zeta\Bigr)_{L^2} \Bigr|\\
&\lesssim \epsilon\ N^{2s}|u_x|_{L^\infty}  |\tilde{P}_N \zeta |^2_{L^2}\\
& \hspace{1cm}+|P_N \zeta|_{L^2} \sum_{N_1\gtrsim N} \Bigl(\frac{N}{N_1}\Bigr)^s  
|P_{N_1}u |_{H^{s+1}} |P_{\lesssim N_1} \zeta |_{L^\infty_x}  \\
 & \lesssim \epsilon\, \delta_N^2 \Bigl( |u_x|_{L^\infty}  | \zeta |^2_{H^s}+|\zeta|_{L^\infty_x}   |u|_{H^{s+1}} |\zeta|_{H^s}\Bigr) \\
 & \lesssim \epsilon\,\delta_N^2 \Bigl( |u_x|_{L^\infty} +|\zeta|_{L^\infty_x}  \mu^{-1/2} \Bigr) | \zeta|_{H^s}(|\zeta|_{H^s}+|u|_{H^{s+1}_\mu}) \; .
 \end{align*}
where in the penultimate step we use the discrete Young convolution inequality. Note however that, as in \eqref{tod1},  we could improve this last inequality in order to loose again a factor 
$\mu^{\frac{1}{2}(s-3/2)-}\lesssim \mu^{-1/2} $ when working at  regularity $ s\in ]1/2,3/2] $ and nothing when $ s>3/2$. But since anyway we loose a factor $ \mu^{-1/2} $ in \eqref{lost} we do not try to get an optimal result here. To complete the bound on this term we notice that for $ N\lesssim 1  $ we directly get 
$$
\epsilon N^{2s} |(\partial_x P_N (u_x \zeta),P_N \zeta) | \lesssim\epsilon |u_x|_{L^2_x} |\zeta|_{L^2_x}^2
$$

To control the third and  fourth  terms in the right-hand side  of \eqref{E3s} we make use of \eqref{prod5}-\eqref{prod6} to get, for any $ N\ge 1 $,
\begin{align}
N^{s} \Bigl(  |[P_N,  \partial_x h_b] u |_{L^2_x}+|[P_N,  h_b] u_x |_{L^2_x}  \Bigr) 
\lesssim \beta \delta_N |b_x|_{W^{s+,\infty}} |u|_{H^s}\; .
\end{align}
It thus remains to estimate the terms in the last line of the right-hand side of \eqref{E3s}.
  For this we first notice   that \eqref{prod6} leads to 
\begin{align}
N^{2s} \Bigl| \Bigl([P_N, h_b ] \zeta_x ,P_N  u \Bigr)_{L^2} \Bigr| 
& \lesssim  \delta_N^2\, \beta |b_x|_{W^{s+,\infty}} |\zeta|_{H^s} |u|_{H^s} \; .
\end{align}
On the other hand, by benefiting of the coefficient $ \mu $, \eqref{prod6}  and \eqref{estut} lead to 
\begin{align}
\mu N^{2s} \Bigl|\Bigl(  \partial_x ([P_N, h_b^3 \partial_x] u_t), P_N u \Bigr)_{L^2}\Bigr| 
& = \mu \, N^{2s} \Bigl|\Bigl(  [P_N, h_b^3 \partial_x] u_t, P_N u_x  \Bigr)_{L^2}\Bigr|\nonumber \\
&  \lesssim   \mu  \, \delta_N^2\, |\partial_x h_b^3|_{W^{s+,\infty}} |u_t|_{H^s} |u_x|_{H^s} \nonumber \\
& \lesssim  \, \delta_N^2\beta |\partial_x h_b^3|_{W^{s+,\infty}}|u_t|_{H^{s}_\mu}  |u|_{H^{s+1}_\mu}\nonumber \\
& \lesssim C(|b_x|_{W^{s+,\infty}})\, \delta_N^2\beta ( 1+\epsilon |u|_{L^\infty_x}  )\nonumber \\
&\hspace{2cm}\times(|\zeta|_{H^{s}}+|u|_{H^s}) |u|_{H^{s+1}_\mu} \; .
\end{align}
for any $ s>1/2 $.\\
Finally we notice that the expression of $ g_b $ in \eqref{defgbeta} (recall that $ h_b=1-\beta b$) leads to 
\begin{align}
 N^s[P_N, g_b] u_t & =   \beta N^s\Bigl(- [P_N, b] u_t+ \mu [P_N, \frac{1}{2} \partial_x(h_b^2 b_x)+\beta b_x^2] u_t\Bigr) 
 \nonumber\\
 & =  B_1+B_2 \label{f1} \; .
  \end{align}
   Applying  \eqref{prod5} and then \eqref{estut} we eventually obtain
$$
|B_2|_{L^2}   \le  \delta_N  \beta \mu |u_t|_{H^s} \lesssim \delta_N \beta \sqrt{\mu} |u_t|_{H^{s}_\mu} 
\lesssim \delta_N \beta \sqrt{\mu}(|\zeta|_{H^s}+ \epsilon |u|_{L^\infty_x}  |u|_{H^s}) \; ,
 $$
 where the implicit constant depends on $ |b_x|_{W^{s+1+,\infty}}$.\\
 To estimate the $ L^2$-norm of $B_1$ we separate the contribution of $P_{\lesssim 1} u_t $ and $ P_{\gg 1} u_t $. For the first contribution 
 we use again \eqref{prod5} and then \eqref{estut}  to get 
 $$
\beta N^s \Bigl|  [P_N, b] P_{\lesssim 1} u_t \Bigr|_{L^2} \lesssim 
\delta_N \beta |P_{\lesssim 1} u_t|_{L^2}  \lesssim \delta_N   \beta (|\zeta|_{H^s}+ \epsilon |u|_{L^\infty_x}  |u|_{H^s})
 $$
 whereas for the second one we use \eqref{prod6} and \eqref{estut}  to get 
 $$
 \beta N^s \Bigl|  [P_N, b] P_{\gg 1} u_t \Bigr|_{L^2} \lesssim 
\delta_N \beta |P_{\gg 1} u_t|_{H^{s-1}}  \lesssim \delta_N  \beta (|\zeta|_{H^s}+ \epsilon |u|_{L^\infty_x}  |u|_{H^s})
 $$
 where the implicit constants depend respectively on $|b|_{W^{s+,\infty}} $ and on $|b_x|_{W^{s+,\infty}}$
Gathering the two above estimates we eventually get, for any $ N\ge 1$, 
\begin{equation}
N^{2s} \Bigl|\Bigl( [P_N, g_b] u_t , P_N u\Bigr)_{L^2}\Bigr| \lesssim 
  \delta_N^2 \beta  (1+ \epsilon |u|_{L^\infty_x})  (|\zeta|_{H^s}^2 +|u|_{H^s}^2)
\end{equation}
that concludes the estimates on the terms appearing in right-hand side of \eqref{E3s}.

Gathering the above estimates,  integrating in time and summing in $ N\ge 1 $,   \eqref{E3s} leads to
\begin{align} 
 \sup_{t\in (0,T)} &E^s_{\mu,\nu} (\zeta(t),u(t)) 
\lesssim   E^s_{\mu,\nu} (\zeta_0,u_0) \nonumber \\
& + T \, (\epsilon \vee \beta)  \Bigl(1+|u|_{L^\infty_{Tx}}+|u_x|_{L^\infty_{Tx}}+ (1+\mu^{-1/2}) |\zeta|_{L^\infty_{Tx}}\Bigr)  
  \sup_{t\in (0,T)} E^s_{\mu,\nu} (\zeta(t),u(t)) 
\label{gf}
\end{align}
for any $0<T< T^{\infty}_{s,\mu,\nu} $, where $ T^{\infty}_{s,\mu,\nu} $  denotes the maximal time of existence of the solution $(\zeta,u) $ 
to \eqref{3l}  in $ (H^{s+1}(\R))^2 $.

According to classical Sobolev inequalities, 
 the local well-posedness of \eqref{3l} in $ (H^{s+1}(\R))^2 $ together with \eqref{gf} ensure that 
for any $ s >1/2 $, $T^\infty_{s,\mu,\nu}= T^\infty_{\frac{1}{2}+,\mu,\nu} $  (the common maximal time of existence of solutions). On the other hand, 
since
$$
1+|u|_{L^\infty_{Tx}}+|u_x|_{L^\infty_{Tx}}+ (1+\mu^{-1/2}) |\zeta|_{L^\infty_{Tx}} \lesssim 1+\mu^{-1/2}  \sup_{t\in (0,T)} E^{\frac{1}{2}+}_{\mu,\nu} (\zeta(t),u(t)) \; ,
$$
\eqref{gf} 
with $ s=\frac{1}{2}+ $ together with a  classical continuity argument ensure that 
\begin{equation}\label{defTinfini}
  T^\infty_{\frac{1}{2}+,\mu,\nu}\gtrsim T_{0}:= C\, 
(\epsilon\vee \beta)^{-1}\Bigl( 1+ \mu^{-1/2}  E^{\frac{1}{2}+}_{\mu,\nu}(\zeta_0,u_0)^{1/2}\Bigr)^{-1}
\end{equation}
where $C $ is independent of $\mu, \, \nu,  \, \epsilon $ and $ \beta$. Moreover,
  for any $s>1/2 $, it holds 
 \begin{equation}
 \sup_{t\in [0,T_0]} E^{s}_{\mu,\nu}(\zeta,u)(t) \lesssim  E^{s}_{\mu,\nu}(\zeta_0,u_0) \; .
\label{defT0}
 \end{equation}
 with an implicit constant that is independent of $\mu, \, \nu,  \, \epsilon $ and $ \beta$.
\subsubsection{$H^{s-1}$ estimate for the difference of two solutions.}\label{sub313}
Let $(\zeta_i,u_i )$ be two solutions to \eqref{3l} with respectively $ \nu_1 $ and $ \nu_2 $, then setting $ \eta=\zeta_1-\zeta_2 $ and $ v=u_1-u_2 $ it holds 
\begin{eqnarray}\label{3ldif}
\left\{
\begin{array}{lcl}
\eta_t-\nu_1 \eta_{txx}+h_b v_x+\epsilon (u_1\eta+v\zeta_2)_x&=&(\nu_1-\nu_2) \partial_t \zeta_{2,xx},\\
 \mfT_b v_t+h_b \eta_x+\frac{\epsilon}{2} h_b ((u_1+u_2) v)_x&=& 0\quad ,\\
\end{array} \right.
\end{eqnarray}
 For $N\ge  1$, applying the operator $P_{N} $  to the equations in (\ref{3}), multiplying respectively by $N^{2(s-1)} P_{ N}  \zeta$ and $ N^{2(s-1)} P_{ N}   v$ the first and the second equation,  integrating with respect to $x$, adding the resulting equations and proceeding as above we get 
\begin{align}
\displaystyle \frac{N^{2(s-1)} }{2} &\frac{d}{dt} \Bigl(|P_{N} \eta|_{L^2}^2+  \nu_1 |P_{N} \eta_x|^2_{L^2}+ (\mfT_b P_{N} v, P_{N} v)_{L^2_x}\Bigr)\nonumber \\
&=-N^{2(s-1)} \frac{\epsilon}{2} \Bigl(   P_{N} ( h_b ((u_1+u_2) v)_x) ,P_{N}  v\Bigr)_{L^2} \nonumber\\
&\quad-N^{2(s-1)} \epsilon \Bigl(   P_{N}  (u_1\eta+v\zeta_2 )_x ,P_{N}  \eta\Bigr)_{L^2} \nonumber\\
&\quad - N^{2(s-1)} ([P_{N} , \partial_x h_b] v , P_{N}  \eta)_{L^2}- N^{2(s-1)}  ([P_{N},  h_b] v_x , P_{N} \eta)_{L^2}\nonumber \\
& \quad-N^{2(s-1)}  \Bigl([P_{\le N}, h_b ] \eta_x - \frac{\mu}{3} \partial_x([P_{N}, h_b^3 \partial_x] v_t) + 
[P_{N}, g_b] v_t   ,P_{N} v\Bigr)_{L^2}\nonumber \\
&\quad + N^{2(s-1)}   (\nu_1-\nu_2)\Bigl( P_{N} \partial_t \zeta_{2,xx}, P_{N} \eta)_{L^2} \; .
\label{E33s}
\end{align}
Let us estimate the terms in the right-hand side one by one. Actually  we will proceed as in the a priori estimates on the solution except that for the difference we do not care of loosing powers of  $ \mu^{-1} $.
First, for $ N\lesssim 1$, it is easy to check that 
\begin{align*} 
N^{2(s-1)}\Bigl|\Bigl(   P_{N} ( h_b ((u_1+u_2) v)_x) ,P_{N}  v\Bigr)_{L^2}\Bigr| &  \lesssim |h_b((u_1+u_2) v)_x)|_{H^{s-1}} |v|_{H^{s-1}}\\
& \lesssim ( |u_1|_{H^s} + |u_2|_{H^s} ) |v|_{H^s} |v|_{H^{s-1}}\\
&\lesssim  \mu^{-\frac{1}{2}} ( |u_1|_{H^s} + |u_2|_{H^s} ) |v|_{H^s_\mu} |v|_{H^{s-1}} \; .
\end{align*}
Now for $ N\gg 1 $ we  rewrite $\frac{1}{2}((u_1+u_2)v)_x  $ as $u_1 v_x +v u_{2,x} $. Since $ s>1/2 $, \eqref{prod2} leads to  
$$
N^{s-1}  |P_N(h_b v u_{2x}) |_{L^2}\le  \delta_N |v u_{2x} |_{H^{s-1}}\lesssim \delta_N \mu^{-\frac{1}{2}}|v|_{H^{s-1}} |u_2|_{H^{s+1}_\mu}\;.
 $$
On the other hand, decomposing $P_N(h_b u_1 v_x) $ as 
$$
P_N(h_b u_1 v_x) = P_N (P_{\ll N} (h_b u_1) \, v_x) + P_N (P_{\gtrsim N} (h_b u_1) v_x)= A_{1,N}'+A_{2,N}' \; ,
$$
we first notice that  \eqref{cm1} and integration by parts lead to 
\begin{align*}
N^{2(s-1)} |(A_{1,N}', P_N v)_{L^2_x}|&  \lesssim \delta_N^2  |(\partial_x(h_b u_1) |_{L^\infty_x} |v|_{H^{s-1}}^2 
\lesssim \delta_N^2  \mu^{-\frac{1}{2}}  |u_1|_{H^{s+1}_\mu} |v|_{H^{s-1}}^2      \;.
\end{align*}
Then  rewriting $ A_{2,N}' $ as 
$$
A_{2,N}'=  P_N \Bigl( \sum_{N_1\gtrsim N} P_{N_1} (h_b u_1) P_{\lesssim N_1} v_x\Bigr)
$$
we observe that \eqref{prod5} leads to 
\begin{align*}
N^{s-1} |A_{2,N}'|_{L^2_x} &\lesssim N^{-1} \sum_{N_1\gtrsim N}
\Bigl(\frac{N}{N_1}\Bigr)^{s}  N_1 |P_{N_1} (h_b u_1) |_{H^{s}} |P_{\lesssim N_1} v |_{L^\infty_x}\\
&  \lesssim N^{-1} \delta_N |(h_b u_1)_x  |_{H^{s}} |v|_{H^{s}} \\
& \lesssim   \delta_N  \mu^{-1}   |u_1|_{H^{s+1}_\mu} |v |_{H^s_\mu}\; .
\end{align*}
Gathering the above estimates we eventually get , 
$$
\epsilon N^{2(s-1)} |(P_N(h_b ((u_1+u_2)v)_x) , P_N v)_{L^2_x}|\lesssim \delta_N^2  \mu^{-1}  |u_1|_{H^{s+1}_\mu} |v|_{H^{s-1}}^2     , \quad \forall N\ge 1 \;.
$$
Let us now tackle the second term in the right-hand side of \eqref{E33s}.
For $ N\lesssim 1 $,  \eqref{prod5} and then \eqref{prod2}  lead to 
\begin{align*} 
N^{2(s-1)}\Bigl| \Bigl(   P_{N}  (u_1\eta+v\zeta_2 )_x ,P_{N}  \eta\Bigr)_{L^2}\Bigr| &  \lesssim |(u_1\eta+v\zeta_2 )_x|_{H^{s-2}} |\eta|_{H^{s-1}}\\
& \lesssim ( |u_1|_{H^s} |\eta|_{H^{s-1}}+ |\zeta_2|_{H^s}  |v|_{H^{s-1}} )|\eta|_{H^{s-1}}\\
 & \lesssim   ( |u_1|_{H^s} + |\zeta_2|_{H^s}  |v|_{H^s_\mu})( |v|_{H^{s-1}}^2+|\eta|_{H^{s-1}}^2)
\end{align*}
where we used that for $ s>1/2$,  $ s+(s-1)=2s-1> (s-1)+\frac{1}{2} $.\\
 Now, for $ N\gg 1$, by using  again \eqref{prod2}  we get 
 $$
N^{s-1} |P_N (v\zeta_2)_x|_{L^2}\lesssim |P_N(v \zeta_2))|_{H^s} \lesssim \delta_N |v \zeta_2|_{H^s} \lesssim \delta_N  |v|_{H^s}  |\zeta_2|_{H^s} 
$$
and 
\begin{align}
\epsilon\, N^{(s-1)} |P_N(u_{1,x} \eta)|_{L^2_x} &\lesssim \epsilon\, \delta_N |u_{1,x} \eta|_{H^{s-1}}  
  \lesssim \epsilon\ \delta_N |u_{1,x}|_{H^{s}} |\eta|_{H^{s-1}} \nonumber \\
  &\lesssim  \epsilon\ \delta_N \mu^{-1/2}|u_1|_{H^{s+1}_\mu}  | \eta|_{H^{s-1}}  \; .\label{lost2}
\end{align}
On the other hand, integrating by parts and using \eqref{cm1} and \eqref{proN2} we obtain
\begin{align*}
\epsilon\, N^{2(s-1)}|(P_N(u_1 \eta_x),P_N  \eta)_{L^2}|&=\epsilon\,N^{2(s-1)} \Bigl|-\frac{1}{2} (P_{\ll N} u_{1,x} P_N \eta, P_N  \eta)_{L^2}  \\
&\quad+([P_N, P_{\ll N} u_1]  \eta_x ,P_N  \eta)_{L^2} 
+ \Bigl(P_N(P_{\gtrsim N} u_1 \,  \eta_x), P_N  \eta\Bigr)_{L^2} \Bigr|\\
 & \lesssim \epsilon\, \delta_N^2 \Bigl( |u_{1,x}|_{L^\infty}  | \eta |^2_{H^{s-1}}+  |u_1|_{H^{s+1}} |\eta|_{H^{s-1}}^2\Bigr) \\
 & \lesssim \epsilon\,\delta_N^2 \mu^{-\frac{1}{2}}  |u_1|_{H^{s+1}_\mu}| \eta|_{H^{s-1}}^2 \; .
 \end{align*}
To control the third and  fourth  terms in the right-hand side  of \eqref{E33s} we make use of \eqref{prod5}-\eqref{prod6} to get 
\begin{align}
N^{s-1} \Bigl(  |[P_N,  \partial_x h_b] v |_{L^2_x}+|[P_N,  h_b] v_x |_{L^2_x}  \Bigr) 
\lesssim \beta \delta_N |v|_{H^{s-1}}\; .
\end{align}
For the fifth term we note   that \eqref{prod6} leads to 
\begin{align}
N^{2(s-1)} \Bigl| \Bigl([P_N, h_b ] \eta_x ,P_N  v \Bigr)_{L^2} \Bigr| 
& \lesssim  \delta_N^2\, \beta |b_x|_{W^{|s-1|+,\infty}} |\eta|_{H^{s-1}} |v|_{H^{s-1}} \; .
\end{align}
Now, according to \eqref{3ldif} and Lemma \ref{proprim'} we  infer that for $ s>1/2 $, 
\begin{align}
| v_t|_{H^{s-1}_\mu} & \lesssim  \Bigl|\partial_x\Bigl(h_b \eta+\frac{\epsilon}{2} h_b (u_1+u_2)v\Bigr)\Bigr|_{H^{s-2}} + \Bigl|\partial_x h_b \eta + \frac{\epsilon}{2} \partial_x h_b (u_1+u_2) v \Bigr|_{H^{s-2}} \nonumber \\
&\lesssim  
   |\eta|_{H^{s-1}}+ \epsilon (|u_1|_{H^s} +|u_2|_{H^s}) |v|_{H^{s-1}} \; .\label{estvt}
   \end{align}
 \eqref{prod6}  and \eqref{estvt} then lead to 
\begin{align}
\mu N^{2(s-1)} \Bigl|\Bigl(  \partial_x ([P_N, h_b^3 \partial_x] v_t), & P_N v  \Bigr)_{L^2}\Bigr| 
 = \mu  N^{2(s-1)} \Bigl|\Bigl(  [P_N, h_b^3 \partial_x] v_t, P_N v_x  \Bigr)_{L^2}\Bigr|\nonumber \\
&  \lesssim   \mu  \delta_N^2\, |\partial_x h_b^3|_{W^{s+,\infty}} |v_t|_{H^{s-1}} |v_x|_{H^{s-1}} \nonumber \\
& \lesssim    \delta_N^2\beta |v_t|_{H^{s-1}_\mu}  |v|_{H^{s}_\mu}\nonumber \\
& \lesssim   \delta_N^2\beta \Bigl(  |\eta|_{H^{s-1}}+ \epsilon (|u_1|_{H^s} +|u_2|_{H^s}) |v|_{H^{s-1}}\Bigr) |v|_{H^{s}_\mu} \; .
\end{align}
for any $ s>1/2 $.\\
It remains to estimate the contribution of $[P_N, g_b] v_t$. 
Clearly \eqref{f1} also holds when substituting $u_t $ by  $ v_t $.   Applying  \eqref{prod5} and then \eqref{estut} we eventually get 
\begin{align*}
 \beta N^{s-1} \mu |[P_N, \frac{1}{2} \partial_x(h_b^2 b_x)+&\beta b_x^2] v_t |_{L^2}
\lesssim  \delta_N  \beta \mu |v_t|_{H^{s-1}} \lesssim  \delta_N  \beta \sqrt{\mu} |v_t|_{H^{s-1}_\mu} \\
&\lesssim \delta_N \beta \sqrt{\mu}(  |\eta|_{H^{s-1}}+ \epsilon (|u_1|_{H^s} +|u_2|_{H^s}) |v|_{H^{s-1}}) \; .
 \end{align*}
 Also  using again \eqref{prod5} and then \eqref{estvt}  we get 
\begin{align*}
\beta N^{s-1}\Bigl|  [P_N, b] P_{\lesssim 1} v_t \Bigr|_{L^2} & \lesssim 
\delta_N \beta |P_{\lesssim 1} v_t|_{L^2} \\
&\lesssim \delta_N   \beta ( |\eta|_{H^{s-1}}+ \epsilon (|u_1|_{H^s} +|u_2|_{H^s}) |v|_{H^{s-1}}),
\end{align*}
 whereas for the second one we use \eqref{prod6} and \eqref{estvt}  to get 
\begin{align*}
 \beta N^{s-1} \Bigl|  [P_N, b] P_{\gg 1} v_t \Bigr|_{L^2} &\lesssim 
\delta_N \beta |P_{\gg1} v_t|_{H^{s-2}} \\
&\lesssim \delta_N  \beta( |\eta|_{H^{s-1}}+ \epsilon (|u_1|_{H^s} +|u_2|_{H^s}) |v|_{H^{s-1}}).
\end{align*}
Gathering all the estimates above we eventually   get
\begin{align}
\displaystyle N^{2(s-1)}  & \frac{d}{dt} \Bigl(|P_N \eta|_{L^2}^2+ \nu_1 |P_N \eta_x|^2_{L^2}+ (\mfT_b P_N v, P_N v)_{L^2_x} \Bigr)\nonumber \\
\lesssim &\; \delta_N^2 \mu^{-1} (1+ |u_1|_{H^{s+1}_\mu} + |u_2|_{H^{s+1}_\mu}+|\zeta_2|_{H^{s}}) ( |v|_{H^s_\mu}^2+|\eta|_{H^{s-1}}^2)\nonumber \\
& 
 + \delta_N^2 |\nu_1-\nu_2|^2 | \partial_t \zeta_{2,xx}|_{L^2}^2 \;  . \label{difdif}
 \end{align} 
 Therefore integrating \eqref{difdif} on $ (0,t) $ for $ 0<t<T $  and summing in $ N\ge 1 $ we get 
 \begin{align} 
   E^{s-1}_{\mu,\nu_1}& (\eta(t),v(t)) \lesssim   E^{s-1}_{\mu,\nu_1} (v(0), \eta(0))+ T |\nu_1-\nu_2|^2 |\zeta_{2,t}|_{L^\infty_T H^{s+1}}^2\nonumber \\
 & + T \mu^{-1}  (1+|{u_1}|_{L^\infty_T H^{s+1}_\mu}+|{u_2}|_{L^\infty_T H^{s+1}_\mu}+|\zeta_2|_{L^\infty_T H^s}) (|v|_{L^\infty_T H^{s}_\mu}^2+|\eta|_{L^\infty_T H^{s-1}}^2).
\label{dif}
\end{align}
\subsection{Local well-posedness of \eqref{BP} }
We will prove the local well-posedness of  \eqref{BP} using a standard compactness method. 
\begin{proposition}[LWP and intermediate time existence for \eqref{BP}]
 \label{essentiel}
Let $ \mu>0 $, $ \varepsilon>0 $, $ \beta>0 $, $s>1/2$ and $ b\in C^{s+2+}_b(\R) $. Assume moreover that \eqref{hb} and Hypothesis \ref{hyp2} are satisfied. 
Then for any  $(\zeta_0,u_0)\in H^s(\R) \times H^{s+1}_\mu(\R) $,  there exists $ T_0=T_0(|\zeta_0|_{H^{\frac{1}{2}+}}+ |u_0|_{H^{\frac{3}{2}+}_\mu})$ verifying 
\begin{equation}\label{defTT0}
 T_0 \gtrsim  (\epsilon\vee \beta)^{-1} (1+\mu^{-1/2}  E^{\frac{1}{2}+}_{\mu,0}(\zeta_0,u_0)^{1/2}\Bigr)^{-1}
\end{equation}
 such that  there exists a solution $(\zeta,u)$ of the Cauchy problem  (\ref{BP}) in  \\ $ C([0,T_0]; H^s(\R)\times H^{s+1}_\mu(\R))$.   This is the unique solution to the IVP \eqref{BP} that belongs to 
 $ L^\infty(]0,T_0[; H^s(\R)\times H^{s+1}_\mu(\R)) $.

Moreover,
$$
|(\zeta,u)|_{L_{T_0}^{\infty}H^s\times H^{s+1}_\mu} \lesssim 
|(\zeta_0,u_0)|_{H^s\times H^{s+1}_\mu}
$$
and  for any $\alpha>0 $, the solution map $S :(\zeta_0,u_0) \longrightarrow (\zeta,u)$ is  continuous  
 from $B(0,\alpha)_{H^s\times H^{s+1}_\mu }  $ into $ C([0,T_0(\alpha)]; H^s(\R)\times H^{s+1}_\mu(\R)) $. 
 Finally, let $ T^* $ be the maximal time of existence in $ H^s(\R) \times H^{s+1}_\mu(\R) $ of the solution  $(\zeta,u)$  emanating from $(\zeta_0,u_0)\in H^s(\R) \times H^{s+1}_\mu(\R) $. Then for any $ 0<T'<T^* $ it holds 
 \begin{align} \label{exp}
  | \zeta|_{L^\infty_{T'} H^s}^2 & + \mu | u|_{L^\infty_{T'} H^{s+1}}^2 \nonumber \\ 
 & \lesssim \exp\Bigl( C\; T'  (\epsilon \vee \beta) \Bigl(1+|u|_{L^\infty_{Tx}}+|u_x|_{L^\infty_{Tx}}+ (1+\mu^{-1/2}) |\zeta|_{L^\infty_{Tx}}\Bigr)\Bigr)E^s_\mu (\zeta_0,u_0)
 \end{align}
 for some universal constant $ C>0 $.
 \end{proposition}
\begin{proof}
$\bullet $ {\it Unconditional uniqueness.} Let $ (\zeta_i,u_i)$, $i=1,2$ be two solution  of the IVP \eqref{BP}  that belong to  $L^\infty(]0,T[; H^s(\R)\times H^{s+1}_\mu)$  for some $ T>0 $.
Setting 
  $ \eta=\zeta_1-\zeta_2 $ and $ v=u_1-u_2 $, exactly the same calculations as in \ref{sub313} on the difference of two solutions  to \eqref{3l}  that lead to \eqref{dif} but with $ \nu_1=\nu_2=0 $   (note that all the computations are justified since for any $ N $,  $ P_N u_i $ and $ P_N \zeta_i $ 
  belong to $ C^1([0,T]; H^\infty)$) lead  for $0<T'<T $ to 
\begin{align} 
|v|_{L^\infty_{T'} H^{s}_\mu}^2+ & |\eta|_{L^\infty_{T'} H^{s-1}}^2  \lesssim   E^{s-1}_{\mu,0} (v(0), \eta(0))\nonumber \\
 & + T'\mu^{-1}    (1+|{u_1}|_{L^\infty_{T} H^{s+1}_\mu}+|{u_2}|_{L^\infty_{T} H^{s+1}_\mu}+|\zeta_2|_{L^\infty_{T} H^s}) \nonumber\\ 
& \hspace{4cm} \times (|v|_{L^\infty_{T'} H^{s}_\mu}^2+|\eta|_{L^\infty_{T'} H^{s-1}}^2)
\label{gfu}
\end{align}
that proves the uniqueness in this class by taking 
$$ 0<T'< \mu(1+|{u_1}|_{L^\infty_{T} H^{s+1}_\mu}+|{u_2}|_{L^\infty_{T} H^{s+1}_\mu}+|\zeta_2|_{L^\infty_{T} H^s})^{-1}
$$
 and  repeating the argument a finite number of times.

\noindent
$\bullet $ {\it Existence.}
Let $(\zeta_0,u_0)\in H^s(\R) \times H^{s+1}_\mu(\R) $. We regularize the initial data by setting 
$\zeta_{0,n}=S_n \zeta $ and $u_{0,n}=S_n u_0 $ where $S_n $ is the Fourier multiplier by  $\chi_{[-n,n]} $. It is straightforward to check that 
 for $ n\ge 1$, $(\zeta_{0,n},u_{0,n})\in (H^\infty(\R))^2 $ with 
 \begin{equation}\label{gd}
    |u_{0,n}|_{H^{s+r}} \le n^r  |u_0|_{H^{s}}\quad\text{and}\quad |\zeta_{0,n}|_{H^{s+r
 }}\le n^r |\zeta_0|_{H^s}\quad \text{for any} \, r\ge 0   \; .
  \end{equation}
 Setting $ \nu= \nu_n=n^{-5} $, we thus obtain that  for any $ s>0 $ and any $ r\ge 0 $
  $$
E^{s+r}_{\mu,\nu_n}(\zeta_{0,n},u_{0,n})=   |\zeta_{0,n}|_{H^{s+r}}^2+n^{-5}  |\partial_x \zeta_{0,n}|_{H^{s+r}}^2+|u_{0,n} |_{H^{s+r+1}_\mu}^2
   \lesssim n^{2r}  E^s_{\mu,0}(\zeta_{0},u_{0})
   $$
   
  In particular setting, for $ s>1/2$,
 \begin{equation} \label{defTs}
  T_s \sim \mu \,  \Bigl(1+|{u_0}|_{H^{s+1}_\mu}+|\zeta_0|_{H^s}\Bigr)^{-1} \le T_0\; , 
  \end{equation}
  we deduce from Subsection \ref{31} and \eqref{defTinfini},  that  we can construct a sequence $ (\zeta_n,u_n)_{n\ge 1} \subset 
  C^1([0,T_s]; (H^\infty(\R))^2) $
   such that for any $ n\ge 1 $, $ (\zeta_n,u_n)$ satisfies \eqref{3l} with $\nu= \nu_n=n^{-5} $.  
  Moreover, from \eqref{defT0}  and \eqref{gd} we infer  that for $ s>1/2 $ and $ r\ge 0 $
  \begin{align}
   \sup_{t\in [0,T_0]} E^{s+r}_{\mu,\nu_n}(\zeta_n,u_n)(t) &  \le 2  E^{s+r}_{\mu,\nu_n}(\zeta_{0,n},u_{0,n})
   \nonumber \\
   & \lesssim n^{2r} E^s_{\mu,0}(\zeta_0,u_0) \; .  \label{313}
 \end{align}
 On the other hand from the first equation in \eqref{3l} and \eqref{prod3} we obtain that on $ [0,T_0] $, 
\begin{align}
|\partial_t \zeta_n|_{H^{s+ 2}}&\le  |(1-\nu_n \partial_x^2)^{-1}\Bigl((h_b u_{n})_x +\epsilon ( \zeta_n u_n)_x \Bigr)|_{H^{s+ 2}}\nonumber\\
& \le |u_{n,x}|_{H^{s+2}}+|(u_n \zeta_n)_x|_{H^{s+2}}\nonumber\\
& \lesssim |u_n|_{H^{s+3}} (1+|\zeta_n|_{L^\infty} )+|\zeta_n|_{H^{s+{3}}} |u_{n}|_{L^\infty}
\nonumber \\
& \lesssim \sqrt{1+E_0^s(u_n,\zeta_n)}\sqrt{ E_0^{s+3} (u_n,\zeta_n)}\lesssim  n^{3}  (1+E_{\mu,0}^s(u_0,\zeta_0)) \; . \label{ou}
\end{align}
For $ n_1\ge n_2 $ applying \eqref{dif} with $ (\zeta_i, u_i)=(\zeta_{n_i},u_{n_i}) $, $i=1,2$, using \eqref{313}-\eqref{ou} and that 
$  |\frac{1}{n_1^{5}}-\frac{1}{n_2^{5}}| \le \frac{1}{n_2^{5}}$  we thus obtain 
\begin{align}
  \|\zeta_{n_1}-\zeta_{n_2}\|_{L^\infty_{T_s} H^{s-1}}^2+\|u_{n_1}-u_{n_2}\|_{L^\infty_{T_s} H^{s}}^2 
  &  \lesssim  E^{s-1}_{\mu,0}(\zeta_{0,n_1}-\zeta_{0,n_2}, u_{0,n_1}-u_{0,n_2}) \nonumber \\
&\hspace{2cm} + \frac{1}{n^{4}_2} \label{315}
\end{align}
where
$$
  T_s \sim \mu \,  \Bigl(1+|{u_0}|_{H^{s+1}_\mu}+|\zeta_0|_{H^s}\Bigr)^{-1} \le T_0 \; .
  $$
 This forces $ ((\zeta_n,u_n))_{n\ge 1}  $ to be a Cauchy sequence in $C([0,T_s]; H^{s-1}\times H^{s}) $.
  Since according to \eqref{defT0},  $((\zeta_n,u_n))_{n\ge 1}$ is bounded in  $C([0,T_0]; H^s\times H^{s+1}_\mu) $,  it follows that there exists $ (\zeta, u) \in 
 L^\infty ([0,T_s]; H^s\times H^{s+1}_\mu) $  such that 
 \begin{eqnarray}
 (\zeta_n, u_n) & \tendsto{n\to  +\infty} & (\zeta,u) \quad \text{in} \; C([0,T_s];H^{s'}\times H^{s'+1}_\mu) , \, \forall 0<s'<s 
 \; . \end{eqnarray}
 that is a solution of the IVP \eqref{BP}. \\
 \noindent
 $\bullet $ {\it Continuity in the strong norm} To prove the continuity of   $(\zeta,u) $ in $ H^s\times H^{s+1}_\mu $ we use Bona-Smith arguments to check that the sequence $ ((\zeta_n,u_n))_{n\ge 1}  $ is actually a Cauchy sequence in $C([0,T_s]; H^s\times H^{s+1}_\mu) $. 
  Let $ n_1\ge n_2\ge 1 $ and set $ (\eta,v)=\zeta_{n_1}-\zeta_{n_2}, u_{n_1}-u_{n_2}) $, $ \nu_i=
  \nu_{n_i}=n_i^{-5} $. By the definition of $(\zeta_n,u_n) $  for any $ 0<r<s $
  \begin{equation}
  E^{s-r}_{\mu,\nu_{n_2}}(\eta(0),v(0))\le n_2^{-2r} E^{s}_{\mu,0}(\eta(0),v(0)) 
  \end{equation}
  Therefore, \eqref{315} together with \eqref{313}  and \eqref{defTs} ensure that 
  \begin{equation}\label{li} 
  \|\eta\|_{L^\infty_{T_s} H^{s-1}}^2+\|v|_{L^\infty_{T_s} H^{s}}^2 \lesssim \frac{1}{n_2^2} E^{s} _{0}(\eta(0),v(0))+\frac{1}{n_2^4}
\le  \Bigl(\frac{1}{n_2}  \gamma(n_2)\Bigr)^2 \; .
  \end{equation}
  with $ \gamma(n) \to 0 $ as $ n\to +\infty $.  On the other hand, \eqref{313} ensures that for any $ r>0 $, 
   \begin{equation}
\sup_{t\in [0,T_{\frac{1}{2}+}]}  E^{s+r}_{\mu,\nu_{n_i}}(\zeta_{n_i}(t),u_{n_i}(t))\lesssim  n_i^{2r} E^{s} _{\mu,0}(\zeta_0,u_0)\; .\label{321}
  \end{equation}
  
  Now observing that $ (\eta,v)$ satisfies \eqref{3ldif} with $ (\zeta_i,u_i)=(\zeta_{n_i} ,u_{n_i}) $ and proceeding as for the obtention of   \eqref{difdif} we eventually get 
  \begin{align}
N^{2s} \frac{d}{dt}E^0_{\mu,\nu_{n_1}}(P_N \eta,P_N v)&
 \lesssim  \delta_N^2 \mu^{-1} (1+ |u_1|_{H^{s+1}_\mu} + |u_2|_{H^{s+1}_\mu}) ( |v|_{H^{s+1}_\mu}^2+|\eta|_{H^{s}}^2)\nonumber\\
 &\quad +\delta_N N^s   |P_N \eta|_{L^2} \Bigl(|v|_{H^{s}}
|\zeta_{n_2}|_{H^{s+1}}+ |v|_{H^{s+1}} |\zeta_{n_2}|_{H^s} \Bigr. \nonumber\\&\Bigl.\hspace{4cm} +n_2^{-5}|\partial_t \zeta_{n_{2},xx}|_{H^s}\Bigr)  \label{dif3}
\end{align}
But in view of \eqref{313} and  \eqref{li}
$$
 \|\zeta_{n_2}\|_{L^\infty_{T_s}H^{s+ 1}} \|v\|_{L^\infty_{T_s} H^s} \lesssim n_2 \frac{1}{n_2} \gamma(n_2)\tendsto{n_2\to +\infty} 0 
 $$
 and \eqref{ou} yields 
 $$
 \frac{1}{n_2^{ 5}} |\partial_t \zeta_{n_2}|_{L^\infty_{T_s}H^{s+2}} \lesssim \frac{1}{n_2^2} (1+E_{\mu,0}^s(u_0,\zeta_0)) \; .
 $$
Integrating in time and summing in $ N $, it thus follows that 
  \begin{align}
   | \eta|_{{L^\infty_{T_s}} H^{s}}^2 +\nu_{n_1}  | \eta|_{{L^\infty_{T_s}} H^{s+1}}^2& +| v|_{{L^\infty_{T_s}} H^{s+1}_\mu}^2
  \le E^s_{\mu,\nu_{n_1}}(\eta(0),v(0))+T_s \tilde{\gamma}(n_2)\nonumber \\
 & \quad+T_s \, \mu^{-1}( 1+ |u_{n_1}|_{L^\infty_{T_s} H^{s+1}_\mu} +|u_{n_2}|_{L^\infty_{T_s} H^{s+1}_\mu}
 +|\zeta_{n_2}|_{L^\infty_{T_s} H^{s}} )  \nonumber \\
 & \qquad \times ( | \eta|_{{L^\infty_{T_s}} H^{s}}^2 +| v|_{{L^\infty_{T_s}} H^{s+1}_\mu}^2)
  \label{difs}
\end{align}
that proves that $ ((\zeta_n,u_n))_{n\ge 1}  $ is indeed a Cauchy sequence in $C([0,T_s]; H^s\times H^{s+1}_\mu) $ and thus 
   $ (\zeta,u) \in C([0,T_s]; H^{s}\times H^{s+1}_\mu) $. 
Observe also that 
$$
E^s_{\mu,\nu_{n_1}}(\zeta_{0,n_1}-\zeta_{0,n_2}, u_{0,n_1}-u_{0,n_2})\tendsto{n_1\to +\infty} E^s_{\mu,0}(\zeta_0-\zeta_{0,n_2},u_0-u_{0,n_2})$$
$$\hspace*{7.5cm}=E^s_{\mu,0}((1-S_{n_2})\zeta_0,(1-S_{n_2})u_0), 
$$
and thus letting $ n_1 \to +\infty $ in \eqref{difs}  we get
 \begin{align}
   \sup_{t\in [0,T_s]} E^{s}_{\mu,0}(\zeta-\zeta_n,u-u_n)(t) &  \lesssim   E^s_{\mu,0}((1-S_n)\zeta_0,(1-S_n)u_0)
    + \tilde{\gamma}(n)  \; .\label{difh2s}
 \end{align}
  $\bullet ${\it Continuity of the flow-map.} 
Let now  $ ((\zeta_{0}^k,u_{0}^k))_{k\ge 1} \subset
H^{s}(\R)\times H^{s+1}_\mu(\R) $ be such that
$
(\zeta_{0}^k,u_{0}^k) \rightarrow (\zeta_0,u_0)
$
in $ H^{s}(\R)\times H^{s+1}(\R) $.
We want to prove that the emanating solution $ (\zeta^k,u^k)$  to \eqref{BP} tends to $ (\zeta,u)
$ in $ C([0,T_s];H^{s}\times H^{s+1}_\mu) $. We set $ \zeta_{0,n}^k=S_n \zeta_0^k$ and $ u_{0,n}^k=S_n u_0^k$ and we call $(\zeta_n^k,u_n^k)\in C([0,T_s];H^{s}\times H^{s+1}_\mu)$ the associated solution to \eqref{3l} with 
 $ \nu=\nu_n=n^{-5}$. 
By the triangle inequality, for $ k $ large enough, it holds  
\begin{align*}
 \sup_{t\in [0,T_s]}  E^{s}_{\mu,0}(\zeta-\zeta^k,u-u^k)(t) 
& \le  \sup_{t\in [0,T_s]}  E^{s}_{\mu,0}(\zeta-\zeta_n,u-u_n)(t) \\
&\qquad + \sup_{t\in [0,T_s]}  E^{s}_{\mu,0}(\zeta_n-\zeta_n^k,u_n-u_n^k)(t) \\
  &\qquad +\sup_{t\in [0,T_s]}  E^{s}_{\mu,0}(\zeta_n^k-\zeta^k,u_n^k-u^k)(t) 
\quad  .
\end{align*}
Using the estimate \eqref{difh2s} on the solution to \eqref{3l} we infer that
\begin{align}
  \sup_{t\in [0,T_s]} E^{s}_{\mu,0}(\zeta-\zeta_n,u-u_n)(t) &+   \sup_{t\in [0,T_s]} E^{s}_{\mu,0}(\zeta^k-\zeta^k_n,u^k-u^k_n)(t)\Bigr)\nonumber \\
 &  \lesssim E^s_{\mu,0}((1-S_{n})\zeta_0,(1-S_{n})u_0) \nonumber \\
 &\qquad +E^s_{\mu,0}((1-S_{n})\zeta_0^k,(1-S_{n})u_0^k)+
 \gamma(n) \label{kak-1}
\end{align}
and thus
\begin{equation}\label{kak1}
\lim_{n\to \infty} \sup_{k\in\N} \Bigl( \sup_{t\in [0,T_s]} E^{s}_{\mu,0}(\zeta-\zeta_n,u-u_n)(t) +   \sup_{t\in [0,T_s]} E^{s}_{\mu,0}(\zeta^k-\zeta^k_n,u^k-u^k_n)(t)\Bigr) =0 \, .
\end{equation}
Therefore, it remains to prove that for any fixed $ n \in \N $, 
\begin{equation}
\lim_{k\to +\infty} \sup_{t\in [0,T_s]}  E^{s}_{\mu,0}(\zeta_n-\zeta_n^k,u_n-u_n^k)(t)=0\label{rem}
\end{equation}
For this we first   notice that \eqref{dif}  with $ \nu_1=\nu_2 $ ensures that
\begin{align}\label{326}
\sup_{t\in ]0,T_s[}  E^{s-1}_{\mu,0}(\zeta_n-\zeta_n^k, u_n-u_n^k)(t)&\lesssim   E^{s-1}_{\mu,0} (\zeta_{0,n}-\zeta_{0,n}^k,u_{0,n}-u_{0,n}^k) \nonumber \\
& \lesssim E^{s-1}_{\mu,0} (\zeta_{0}-\zeta_{0}^k,u_{0}-u_{0}^k)\; .
\end{align}
and that \eqref{321} leads for $ r\ge 0  $ to 
\begin{equation}\label{327}
\sup_{t\in [0,T_{\frac{1}{2}+}]}  E^{s+r}_{\mu,0}(\zeta_{n}^k(t),u_{n}^k(t))\lesssim  n^{2r} E^{s} _{\mu,0}(\zeta_{0,n}^k,u_{0,n}^k)
\lesssim  n^{2r} (E^{s} _{\mu,0}(\zeta_0,u_0)+1)\; .
\end{equation}
Now, setting $ (\eta,v)=(\zeta_n-\zeta_n^k,u_n-u_n^k) $,  observing that $ (\eta,v)$ satisfies \eqref{3ldif} with $ (\zeta_1,u_1)=(\zeta_{n} ,u_{n}) $,
  $ (\zeta_2,u_2)=(\zeta_{n}^k,u_{n}^k) $ and $ 
\nu_1=\nu_2=n^{-5} $  and proceeding as in \eqref{dif3}  we 
 get 
    \begin{align}
N^{2s} \frac{d}{dt}E^0_{\mu, \nu_{n}}(P_N \eta,P_N v)&
 \lesssim   \delta_N^2 \mu^{-1} (1+ |u_n|_{H^{s+1}_\mu} + |u_n^k|_{H^{s+1}_\mu}) ( |v|_{H^{s+1}_\mu}^2+|\eta|_{H^{s}}^2)\nonumber\\
 & +\delta_N N^s   |P_N \eta|_{L^2} \Bigl(|v|_{H^{s}}
|\zeta_{n}^k|_{H^{s+1}}+ |v|_{H^{s+1}} |\zeta_{n}^k|_{H^s}\Bigr)  \; .  \label{dif33}
\end{align}
But \eqref{326}-\eqref{327} ensure that 
$$
 |v|_{H^{s}} |\zeta_{n}^k|_{H^{s+1}}\lesssim 
 n \, \mu^{-\frac{1}{2}}\Bigl[ (E^{s} _{\mu,0}(\zeta_0,u_0)+1)E^{s-1}_{\mu,0} (\zeta_{0}-\zeta_{0}^k,u_{0}-u_{0}^k)\Bigr] ^{1/2}\; .
$$ 
Therefore integrating in time and summing in $ N\ge 1 $, it  follows that 
  \begin{align}
   \| \eta\|_{{L^\infty_{T_s}} H^{s}}^2 +& \| v\|_{{L^\infty_{T_s}} H^{s+1}_\mu}^2  \lesssim E^s_{0}(\eta(0),v(0))+T_s \mu^{-1}  n^2 (E^{s} _{0}(\zeta_0,u_0)+1) \nonumber \\
   & \hspace{1cm} \times E^{s-1}_{0} (\zeta_{0}-\zeta_{0}^k,u_{0}-u_{0}^k)+T_s \mu^{-1}( 1+ \|u_{n}\|_{L^\infty_{T_s} H^{s+1}} \nonumber \\
   & \hspace{1.5cm} +\|u_{n}^k\|_{L^\infty_{T_s} H^{s+1}}
 +\|\zeta_{n}^k\|_{L^\infty_{T_s} H^{s}} ) 
  ( \|\eta\|_{{L^\infty_T} H^{s}}^2 
  +\| v\|_{{L^\infty_T} H^{s+1}_\mu}^2),
  \label{difs33}
\end{align}
which  ensures that 
\begin{align*}
\|\eta\|_{{L^\infty_{T_s}} H^{s}}^2 + 
 \|v\|_{{L^\infty_{T_s}} H^{s+1}_\mu}^2 &\lesssim 
  E^s_{0}(\eta(0),v(0))+T_s \,  \mu^{-1}  n^2 (E^{s} _{0}(\zeta_0,u_0)+1) \\
  & \hspace{4cm} \times E^{s-1}_{0} (\zeta_{0}-\zeta_{0}^k,u_{0}-u_{0}^k),
\end{align*}
  and proves \eqref{rem}.  Combining \eqref{kak1} and \eqref{rem}, we thus obtain
 the continuity of the flow map in  $ C([0,T_s]; H^s\times H^{s+1}_\mu)$. Hence the IVP \eqref{BP} is locally well-posed in $H^s(\R) \times H^{s+1}_\mu(\R) $ with a minimal time of existence $ T_s$ that satisfies \eqref{defTs}. Estimates  \eqref{defTinfini}-\eqref{defT0} then force $$
 T_s\gtrsim (\epsilon\vee \beta)^{-1} (1+\mu^{-1/2}  E^{\frac{1}{2}+}_{\mu,0}(\zeta_0,u_0)^{1/2}\Bigr)^{-1}\; .$$
 Finally, let $(\zeta_0,u_0) \in H^s\times H^{s+1}_\mu $ and $ T^*_s $ be the maximal time of existence in $ H^s\times H^{s+1}_\mu $ of  the emanating solution $(\zeta,u) $. Then proceeding exactly as in the obtention of \eqref{gf} in the preceding subsection we get for any $0<t_0<t_0+\Delta t<T' <T^*_s$, 
 \begin{align} 
 \| \zeta\|_{L^\infty(]t_0, t_0+\Delta t[; H^s)}^2 &+ \| u\|_{L^\infty(]t_0, t_0+\Delta t[; H^{s+1}_\mu)}^2  \lesssim  E^s_\mu (\zeta(t_0),u(t_0))\nonumber \\
 &+ \Delta t   \,  (\epsilon \vee \beta)  \Bigl(1+\|u\|_{L^\infty_{T'x}}+\|u_x\|_{L^\infty_{T'x}}+ (1+\mu^{-1/2}) \|\zeta\|_{L^\infty_{T'x}}\Bigr)  
  \nonumber \\
 & \times ( \| \zeta\|_{L^\infty(]t_0, t_0+\Delta t[; H^s)}^2+ \| u\|_{L^\infty(]t_0, t_0+\Delta t[; H^{s+1}_\mu)}^2)
\label{gfgf1}
\end{align}
Therefore, for $ \Delta t \sim   (\epsilon \vee \beta)^{-1}  \Bigl(1+\|u\|_{L^\infty_{Tx}}+\|u_x\|_{L^\infty_{Tx}}+ (1+\mu^{-1/2}) \|\zeta\|_{L^\infty_{Tx}}\Bigr)  ^{-1} $, it holds 
$$
 \| \zeta\|_{L^\infty(]t_0, t_0+\Delta t[; H^s)}^2 + \| u\|_{L^\infty(]t_0, t_0+\Delta t[; H^{s+1})}^2  \lesssim E^s_{\mu,0} (\zeta(t_0),u(t_0))
 $$
 This proves \eqref{exp} by dividing $[0,T'] $ in small intervals of length 
 $$
 \Delta t \sim  \min\Bigl[ T',\; (\epsilon \vee \beta)^{-1}  \Bigl(1+\|u\|_{L^\infty_{T'x}}+\|u_x\|_{L^\infty_{T'x}}+ (1+\mu^{-1/2}) \|\zeta\|_{L^\infty_{T'x}}\Bigr)  ^{-1} \Bigr]\; .
  $$
 \end{proof}

\section{A priori estimates and global existence of strong solutions to \eqref{BPW}}
In this section, we establish the global existence for any fixed $\mu,\,\epsilon$ and $\beta\ge 0$ of (\ref{BPW}) always under hypothesis $\eps=O(\mu)$ and $\beta=O(\mu)$.
This leads to the proof of Theorem \ref{maintheorem}.

Recall that $\zeta$ and $u$ are solution of system (\ref{BPW})
\begin{eqnarray}
\left\{
\begin{array}{lcl}
\zeta_t+(hu)_x &=&0,\\
u_t +\zeta_x+\eps uu_x- \frac{\mu}{3}u_{txx}&=&0.
\end{array} \right.
\nonumber
\end{eqnarray}
This system is simpler than \eqref{BP} and thus proceeding in the same way as in the preceding section it is clear that we can get a LWP result. Before stating this result, we  would like to make some remarks. First note that this time Hypothesis \ref{hyp2} is not needed on $ \beta $ since $1+\mu \toh $ is replaced by $ (1-\mu\partial_x^2) $ that does not depend on $ \beta $ and has an algebraic inverse that is continuous from $H^s(\R) $ to $H^{s+1}_\mu(\R) $. Second, we do not need to multiply \eqref{BPW}$_2$ by $ h $ or $ h_b $ here since we do not attempt to get a long time existence result (our goal is to prove a global existence result by deriving suitable a priori bounds thus a local existence result is sufficient). Therefore looking at the Bona-Smith approximation system associated with \eqref{BPW} and using the trivial estimates 
$$
N^{2s}\beta  |(P_N(b u)_x, P_N \zeta)|\lesssim\beta \delta_N^2\, |bu|_{H^{s+1}}|\zeta|_{H^s}
\lesssim C(|b|_{H^{s+1}})\beta \, \mu^{-1/2} |u|_{H^{s+1}_\mu} |\zeta|_{H^s}
$$
and 
$$
N^{2s}  |(P_N \zeta _x, P_N u)|\lesssim \delta_N^2\,  \mu^{-1/2} |u|_{H^{s+1}_\mu} |\zeta|_{H^s}
$$
we can get the following a priori estimate 
\begin{align} 
 \sup_{t\in (0,T)}& E^s_{\mu,\nu} (\zeta(t),u(t)) 
\lesssim   E^s_{\mu,\nu} (\zeta_0,u_0) + \nonumber \\
& + T\Bigl[\mu^{-1/2}+(\epsilon \vee \beta)  \Bigl(1+|u|_{L^\infty_{Tx}}+|u_x|_{L^\infty_{Tx}}+ (1+\mu^{-1/2}) |\zeta|_{L^\infty_{Tx}}\Bigr)  \Bigr] \nonumber \\
& \hspace{6cm} \times \sup_{t\in (0,T)} E^s_{\mu,\nu} (\zeta(t),u(t)) 
\label{gfgf}
\end{align}
for any $0<T< T^{\infty}_{s,\mu,\nu} $, where $ T^{\infty}_{s,\mu,\nu} $  denotes the maximal time of existence  solution $(\zeta,u) $ 
to the Bona-Smith approximation of \eqref{BPW}  in $ (H^{s+1}(\R))^2 $. This estimate combining with a similar estimate on the difference of two solutions in $H^{s-1}(\R) \times H^s_\mu(\R) $ and the same types of consideration as in the preceding section lead to the following LWP result for \eqref{BPW}.
\begin{proposition}[LWP for \eqref{BPW}]
 \label{LWPBPW}
Let $ \mu>0 $, $ \varepsilon>0 $, $ \beta>0 $, $ s>1/2$ and $ b\in H^{s+1}(\R) $. 
Then for any  $(\zeta_0,u_0)\in H^s(\R) \times H^{s+1}_\mu(\R) $,  there exists $ T_0=T_0(|\zeta_0|_{H^{\frac{1}{2}+}}+ |u_0|_{H^{\frac{3}{2}+}_\mu})$ verifying 

$$
 T_0 \gtrsim  [\mu^{-1/2}+(\epsilon \vee \beta)  \Bigl(1+\mu^{-1/2}
 (|\zeta_0|_{H^{\frac{1}{2}+}}+ |u_0|_{H^{\frac{3}{2}+}_\mu})\Bigr)  \Bigr]^{-1}
$$
 such that  there exists a solution $(\zeta,u)$ of the Cauchy problem  (\ref{BPW}) in \\ $ C([0,T_0]; H^s(\R)\times H^{s+1}_\mu(\R))$ that  verifies estimate  (\ref{gfgf}).   This is the unique solution to the IVP \eqref{BPW} that belongs to 
 $ L^\infty(]0,T_0[; H^s(\R)\times H^{s+1}_\mu(\R)) $ . 

 \end{proposition}

To obtain the uniform estimates, we proceed as in \cite{Sch} by constructing a convex positive entropy for the associated hyperbolic system.  To do so, we transform  slightly  the system  (\ref{BPW}) by multiplying the first equation (\ref{BPW})$_1 $by 
$\epsilon$ and  rewriting the system in terms of $(h,u)$ as new unknown. Unfortunately this operation will harmlessly influence the uniformity of the estimation on $u$, with respect of $\epsilon.$  However, the obtained estimation is enough to deduce the globalism of our solution.

System (\ref{BPW}) becomes:

\begin{eqnarray}\label{2b}
\left\{
\begin{array}{lcl}
h_t+(\epsilon hu)_x &=&0,\\
u_t +(\frac{1}{\eps}h+\frac{\eps}{2} u^2 )_x- \frac{\mu}{3}u_{txx}+\frac{\beta}{\eps}b_x&=&0.
\end{array} \right.
\end{eqnarray}

The hyperbolic part of the system is given by  
\begin{eqnarray}\label{3}
\left\{
\begin{array}{lcl}
h_t+(\epsilon hu)_x &=&0,\\
u_t +(\frac{1}{\eps}h+\frac{\eps}{2} u^2 )_x &=&0.
\end{array} \right.
\end{eqnarray}
Let us we recall the notion of entropy for a hyperbolic system. Consider the system 
\begin{eqnarray}\label{H1} V_t+f(V)_x=0, \end{eqnarray}
where $V=V(t,x) \in \R^n$, $f: \R^n \longrightarrow \R^n$ a smooth function. We say that a pair of  functions 
$\eta, q\; : \R^n\to \R $  is an entropy-entropy flux pair if all smooth solutions of  (\ref{H1}) satisfy the additional conservation law
\begin{eqnarray}\label{H2} \eta(V)_t+q(V)_x=0, \end{eqnarray}
which can also be written
$$ \nabla \eta  V_t + \nabla q V_x=0.$$
On the other hand, multiplying (\ref{H1}) by $\nabla \eta$, we obtain
$$  \nabla \eta V_t+ \nabla\eta \nabla f  V_x =0.$$
This ensures that   the compatibility condition
\begin{eqnarray}\label{H3}\nabla \eta \nabla f=\nabla q, \end{eqnarray}
 forces any smooth solutions of  (\ref{H1}) to satisfy the additional conservation law \eqref{H2}.  We recall
$ h=1+\eps \zeta -\beta b,$  and defining $\sigma(h)= h\ln h,\, \sigma_L(h)=\sigma(1)+\sigma'(1)(h-1)=h-1 $
and $ \sigma_0 \; :\; \R_+ \to  \R $ by
$$
\sigma_0(x)=\left\{
\begin{array}{rcl}\sigma(x)-\sigma_L(x)=x\ln x -x+1 & \text{if} & x>0 \\
 1  & \text{if} & x=0
 \end{array}
 \right.
 .$$ 
Note that $\sigma_0$ is a convex function on $ [0,+\infty[ $ and  enjoys the following property.
\begin{lemma} \label{rere}
 Let $ s>1/2 $. The functional 
$$
 \zeta \mapsto \int_\R\sigma_0(1+\eps\zeta-\beta b)dx  $$
is well defined and continuous from the set 
$$
\Theta:=\{\zeta\in H^{s} (\R) , 1+\eps\zeta -\beta b \geq 0 \; {\rm{on} }\; \R\} \; .
$$
of $ H^s(\R) $ into $\R $. Moreover, there exists  $ C>0 $ such that for all $ \zeta \in \Theta $, 
\begin{equation}\label{boundd}
0 \le  \int_\R\sigma_0(1+\eps\zeta-\beta b)dx \le C \int_\R (\eps^2\zeta^2+\beta^2 b^2) dx \; .
\end{equation}
\end{lemma}
\begin{proof} 
First it is easy to check that $ \sigma_0 \ge 0 $ on $ \R_+ $. Moreover, 
    as $\sigma_0(1+x) \sim x \ln x $ at $ +\infty $ and 
$\sigma_0(1+x) \sim x^2 $ near the origin,  
there exists $ M> 1  $ and $c^M_1, \, c^M_2 >0 $ such that  
\begin{eqnarray}
(c_1^M)^{-1} x^2\ge 
\sigma_0(1+x) \ge c^M_1 x^2 \;  &  \text{for} & \; -1\le x<M \nonumber \\
\quad \text{and}\quad (c^{M}_2)^{-1} x^2 \ge \sigma_0(1+x) \ge c^{M}_2 x \ln x \ge x   & \text{ for }  &x\ge M \; .\label{estA} 
\end{eqnarray}
This clearly leads to \eqref{boundd} and proves that $\zeta \mapsto \int_\R\sigma_0(1+\zeta -\beta b)dx  $ is well defined for $ \zeta \in L^2(\R) $ when $b\in L^2(\R) $.
Now, since $ H^s(\R)\hookrightarrow C(\R)$ for $ s>1/2$, the continuity of $ \zeta \mapsto \int_\R\sigma_0(1+\eps \zeta-\beta b)dx  $ from $ \Theta $ into $ \R$, as soon as $b\in H^s(\R)$,  follows from \eqref{rere} together with the Lebesgue convergence theorem.
\end{proof}
We introduce the Orlicz class associated to the function $\sigma_0(1+\cdot-\beta b)$ 
\begin{equation}\label{Orlicz}
\Lambda_{\sigma_0}:= \left\{\zeta\; \mbox{measurable} \;/ \eps \zeta\ge -1+\beta b  \; \text{and}\;  \displaystyle \int_{\R}\sigma_0(1+\eps\zeta(x)-\beta b)\,dx<+\infty \right\},
\end{equation} 
with the notation $|\zeta|_{\Lambda_{\sigma_0}}:= \displaystyle{\int_{\R}  \sigma_0(1+\eps\zeta(x)-\beta b(x))\,dx}.$\\
Now, we shall establish a uniform crucial entropic estimate for our solution.
\begin{proposition}\label{thm1} 
Let $(\zeta_0,u_0)\in  H^s\times H^{s+1}_\mu$, $b\in H^{s+1}$  for $s>1/2$, and such that $1+\eps\zeta_0-\beta b>0$. Then, the solution  
$(\zeta,u)\in C([0,T_0]; H^s\times H^{s+1}) $ to \eqref{BPW}, constructed in Proposition \ref{LWPBPW}, satisfies 
$ 1+\eps\zeta-\beta b> 0 $  in $[0,T_0]\times \R$ with $\zeta\in L^\infty(]0,T_0[; \Lambda_{\sigma_0}) $  and the following inequality holds true:
\begin{eqnarray} 
\begin{array}{l}
 \frac{1}{2} |u(t)|_{H_{\mu}^1}^2 + 
 \frac{1}{\eps^2} \int_{-\infty}^{+\infty}  \sigma_0(1+\eps\zeta(t,x)-\beta b(x))dx  \le  
 \\
\hspace{1cm}  \left( |u_0|_{H_{\mu}^1}^2+\frac{1}{\eps^2} \int_{-\infty}^{+\infty}  \sigma_0(1+\eps\zeta_0(x)-\beta b(x))dx + (\frac{\beta}{2\eps})^2 |b_x|_{{\rm L}^2}^2 \right)e^t, \\
\hfill \forall t\in [0,T_0] .\label{inq1}
\end{array}
\end{eqnarray}
\end{proposition} 


\begin{proof} 
We first assume that  $(\zeta_0,u_0) \in  (H^\infty(\R)) \times H^\infty(\R)$. 

According to Proposition \ref{LWPBPW}, \eqref{BPW} has got a unique solution $(\zeta,u)\in C([0,T_0]; H^\infty\times H^{\infty}) $ emanating from  $(\zeta_0,u_0)$, where $ T_0 $ only depends on $ |\zeta_0|_{H^{\frac{1}{2}+} }+
  |u_0|_{H^{\frac{3}{2}+}} $.

We return now to the system (\ref{2b}) 
\begin{eqnarray}\label{4}
\left\{
\begin{array}{lcl}
h_t+(\epsilon uh)_x &=&0,\\
u_t+(\frac{1}{\eps}h+\eps u^2/2)_x&=&\frac{\mu}{3}u_{xxt}-\frac{\beta}{\eps}b_x ,
\end{array} \right.
\end{eqnarray}
with initial data $h(0,x)=h_0(x)=1+\eps\zeta_0(x)-\beta b(x)$ and $u(0,x)=u_0(x).$ The associated hyperbolic system becomes 
\begin{eqnarray}\label{5}
\left\{
\begin{array}{lcl}
h_t+(\eps uh)_x &=&0,\\
u_t+(\frac{1}{\eps}h+\eps u^2/2)_x&=& 0.
\end{array} \right.
\nonumber
\end{eqnarray}
As in (\ref{H1}), let $f:\R^2 \longrightarrow \R^2$ be defined by 
$$ f(h,u)=(\eps hu,\frac{1}{\eps} h +\eps u^2/2).$$
Let $\eta(\cdot,\cdot) $ and $q(\cdot,\cdot)$  be a pair of functions satisfying the compatibility condition (\ref{H3}). Then,  setting $ V=(h,u)^T $, the solution of (\ref{4}) satisfies  
\begin{eqnarray}\label{H5} 
\eta(h,u)_t+ q(h,u)_x &= & \nabla \eta(V) V_t +\nabla q(V) V_x\nonumber \\
& = & \nabla \eta\Bigl( V_t + (f(V))_x\Bigr)\nonumber\\
& = &\nabla \eta (V) \Bigl(0, \frac{\mu}{3}u_{xxt}-\frac{\beta}{\eps}b_x  \Bigr)^T\nonumber \\
& = & \frac{\partial \eta}{\partial u} (\frac{\mu}{3}u_{xxt}-\frac{\beta}{\eps}b_x).  \end{eqnarray}
Let $\eta $ be the function of the form
$$ \eta(h,u)=u^2/2+\frac{1}{\eps^2}\alpha(h),$$
for some function $\alpha$. Thus, (\ref{H5}) becomes
\begin{eqnarray}\label{H6}\begin{array}{lcl} \eta(h,u)_t+ q(h,u)_x&=&  \frac{\mu}{3}u_{xxt} u-\frac{\beta}{\eps}b_x u.\\
&=&-\frac{\mu}{3}(u_x^2/2)_t+ \frac{\mu}{3}(uu_{xt})_x-\frac{\beta}{\eps}b_x u . \end{array}\end{eqnarray}
In order to get an a priori estimate on solutions to \eqref{4}, we have to choose $\alpha$ to be a convex function  and  $ \eta$ to be a convex  and positive function. Since $ u\mapsto u^2/2 $ is convex and positive, it actually suffices to ask $ \alpha $ to be also convex and positive.   At this stage, it is worth noticing that Lemma \ref{lemminmax} in the Appendix ensures that 
 $\inf_{[0,T_0]\times \R} 1+\eps\zeta -\beta b>0 $. 

We set $\alpha(h)=\sigma(h) =h\ln h$ that is a convex function on 
 $ ]0,+\infty[ $. 
It is straightforward to check that with this $ \alpha $,  $\eta $ satisfies the compatibility condition \eqref{H3} with the entropy flux given by 
$$
q(h,u)=\frac{1}{\eps}\alpha(h) u + \frac{1}{\eps} h u  +\eps  u^3/3 \; .
$$
Thus we have found an entropy which is convex but not positive. To obtain a positive convex entropy $\eta $, it suffices to substract from $\alpha $ its linear part at $ 1 $ that leads to 
$$
\tilde{\alpha}(h)= \alpha(h) - \alpha'(1) (h-1)= h \ln h -h+1=\sigma_0(h) \; ,
$$
so that the  entropy function becomes
\begin{equation}\label{defen}
\tilde{\eta}(h,u)=u^2/2+\frac{1}{\eps^2}\sigma_0(h) \; .
\end{equation}
Note that, in order for \eqref{H3} to hold, the entropy flux function $ q$ has to be modified consequently and becomes
$$ 
\begin{array}{lcl}
\tilde{q}(h,u) & = &  q(h,u)-q(1,0)-(\Frac{1}{\eps^2}\alpha'(1),0)[f(h,u)-f(1,0)]\;,\\
& = & \displaystyle \frac{1}{\eps}\alpha(h)u + \eps\frac{u^3}{3}\; ,
\end{array}
$$
where we choose the constant so that $ \tilde{q}(1,0)=0 $.
Now, by omitting the tilde, the new $\eta$ and $q$ satisfy the equation (\ref{H6}) which will be the starting point of our calculations. As in \cite{Sch}, we consider the space time square
$$ R= \left\{ (s,x)\in \R^2/\, 0\leq s\leq t , \, -N\leq x\leq N \right\}, \, \mbox{ for } N  >0.$$
Then  integrating equation (\ref{H6}) over $R$ and using the divergence theorem, we get
\begin{eqnarray*}\begin{array}{lcl} 
& \displaystyle
\int_{-N}^{N}& (\eta(h,u)(t,x)-\eta(w,u)(0,x))dx +\displaystyle \int_0^t (q(h,u)(s,N)-q(h,u)(s,-N))ds \\ \\
& &= \displaystyle-\frac{\beta}{\eps}\int_0^t \int_{-N}^N b_x(x) u(s,x)dx\, ds 
\displaystyle -\frac{\mu}{3}\int_{-N}^{N}(u_x^2/2(t,x)- u_x^2/2(0,x))dx  \\ \\ & &\hspace*{3cm}\displaystyle +\frac{\mu}{3}\int_0^t (uu_{xt}(s,N)-uu_{xt}(s,-N))  ds.  \\ \\ \end{array}\end{eqnarray*}

  The regularity on $\zeta,\,u, b$ and the fact that,by lemma (\ref{rere}),  $ \sigma_0(h(t)) \in L^{1}(\R) $ for any $ t\in [0,T_0] $ we Letting $N$ go to $+\infty$ in the above equality, we can thus use  the Lebesgue dominated convergence theorem to get 
\begin{eqnarray}\label{H77}
\begin{array}{lcl}
& &\displaystyle \int _{-\infty}^{+\infty}\eta(h,u)(t,x)dx +\frac{\mu}{3}\int _{-\infty}^{+\infty}\frac{u_x^2}{2}(t,x) dx\\  \\& &=\displaystyle 
\int _{-\infty}^{+\infty}\eta(h,u)(0,x)+\frac{\mu}{3}\int _{-\infty}^{+\infty}\frac{u_x^2}{2}(0,x) dx-\frac{\beta}{\epsilon}\int_0^t \int_{\R} b_x(x) u(s,x)dx ds. 
\end{array}
\end{eqnarray}

Equality (\ref{H77}) leads to 
\begin{eqnarray}\label{H78}
\begin{array}{lcl}
& &\displaystyle |u(t)|_{L^2}^2 +\frac{\mu}{3}|u_x(t)|_{L^2}^2 +\frac{2}{\eps^2}\int _{-\infty}^{+\infty}\sigma_0(h(t,x)) dx\\  \\& &\qquad \leq\displaystyle |u_0|_{L^2}^2 +\frac{\mu}{3}|u_{0x}|_{L^2}^2 +\frac{2}{\eps^2}\int _{-\infty}^{+\infty}\sigma_0(h_0(x)) dx\\  \\& & \hspace*{3cm} +\displaystyle (\frac{\beta}{\eps})^2 |b_x|_{L^2}^2\,t +\int _{0}^{t}|u(s)|_{L^2}^2 ds
\end{array}
\end{eqnarray}
This implies that $y(t)=|u(t)|_{L^2}^2$ verifies the following differential inequality : 

$$ \displaystyle y(t)\leq \phi(t) + \int_0^t y(s)ds, $$ where $\phi$ is define by :
$$
\begin{array}{lcl}
\phi(t) & =& |u_0|_{L^2}^2 +\frac{\mu}{3}|u_{0x}|_{L^2}^2 +\frac{2}{\eps^2}\int _{-\infty}^{+\infty}\sigma_0(h_0(x)) dx + (\frac{\beta}{\eps})^2 |b_x|_{L^2}^2\,t \\ 
& =& Z_0 +(\frac{\beta}{\eps})^2 |b_x|_{L^2}^2\,t
\end{array}
$$
Now using Gronwall's lemma we get :
$$ \displaystyle \int_0^t y(s)ds \leq y(0)(e^t-1)+Z_0(e^t-t-1)+(\frac{\beta}{\eps})^2 |b_x|_{L^2}^2 (e^t-\frac{t^2}{2}-t-1). $$
Replace this bound in (\ref{H78}) and a  simple majorations we obtain :
\begin{eqnarray}\label{H79}
\begin{array}{lcl}
& &\displaystyle |u(t)|_{L^2}^2 +\frac{\mu}{3}|u_x(t)|_{L^2}^2 +\frac{2}{\eps^2}\int _{-\infty}^{+\infty}\sigma_0(h(t,x)) dx\\  \\& &\hspace*{2cm} \leq\displaystyle \left(2Z_0+(\frac{\beta}{\eps})^2 |b_x|_{L^2}^2\right)\,e^t.
\end{array}
\end{eqnarray}

  This proves \eqref{inq1} for $(\zeta_0,u_0)\in H^\infty(\R) \times H^\infty(\R)$. The result for 
   $(\zeta_0,u_0)\in (H^s(\R)\times H^{s+1}(\R)) $ follows by  using the continuity of the flow-map together with  Lemma  \ref{rere}.  
   \end{proof}
Next, we state the global well-posedness  result.  
\begin{proposition}\label{prop1} Let  $(\zeta_0,u_0)\in H^s(\R)\times H^{s+1}(\R)$ and $b\in H^{s+1},\, s>1/2$, such that 
 $1+\eps\zeta_0-\beta b >0$ . Then the unique solution $(\zeta,u) $ to  \eqref{BPW} constructed in Proposition \ref{essentiel} 
  can be extended for all positive times and thus belongs to $ C(\R_+;H^{s}\times H^{s+1})$. Moreover, for any $ T>0 $ there exists a constants $ C_{T,s}>0 $ only depending on $|\zeta_0|_{H^s},\, |u_0|_{H^{s+1}},\, |b|_{H^{s+1}}$ and the parameters $\mu,\, \eps$ and $\beta$ such that 
  \begin{equation}
  |\zeta|_{L^\infty(]0,T[;H^s)}+  |u|_{L^\infty(]0,T[;H^{s+1})}\le C_{T,s}
  \end{equation}
 and the flow-map $S :(\zeta_0,u_0) \longrightarrow (\zeta,u)$ is  continuous  
 from $H^s\times H^{s+1}  $ into $ C([0,T]; H^s(\R)\times H^{s+1}(\R)) $.
  \end{proposition}
  \begin{proof}
  According to \eqref{gfgf}  and the local well-posedness result, it suffices to proves that for any $ T>0 $ there exists $ c_T>0 $ only depending on $ T $, $|\zeta_0|_{H^s},\, |u_0|_{H^{s+1}},\, |b|_{H^{s+1}}$ and the parameters $\mu,\, \eps$ and $\beta$ , such that 
  if the solution $(\zeta,u) $ to \eqref{BPW} belongs to $ C([0,T[;H^s\times H^{s+1})$ then 
 \begin{equation}
 |\zeta|_{L^\infty(]0,T[\times \R)}+|u_x|_{L^\infty(]0,T[\times \R)}\le c_T \; .\label{estest}
 \end{equation}
   We mainly follow the proof of Theorem $1.2$ in \cite{Ami}.  Let $N$ be a positive odd integer, we start by deriving an estimate on 
  $ \sup_{t\in [0,T[} |\zeta(t)|_{L^N} $.
   For this we multiply the first equation of (\ref{BPW}) by $\zeta^N$ and integrate with respect to $x$, to get
$$\frac{1}{N+1}\frac{d}{dt}\int_{\R} \zeta^{N+1} =-\int_{\R}  \zeta^N ((1-\beta b)u)_x -\eps\frac{N}{N+1}\int_{\R}  \zeta^{N+1}u_x.
$$
Therefore integrating the above identity on $(0,t) $, using that $ N+1$ is an even integer, we get 
\begin{equation} \frac{1}{N+1}|\zeta(t) |_{L^{N+1}}^{N+1}=  \frac{1}{N+1}|\zeta_0 |_{L^{N+1}}^{N+1}
 -\int_0^t \int_{\R}  \zeta^N ((1-\beta b)u)_x -\eps \frac{N}{N+1}\int_0^t \int_{\R}  \zeta^{N+1}u_x.\label{LM}
\end{equation}
Now, we make use of the fact that for any  $ f\in L^2(\R) $ it holds 
 $(1-\frac{\mu}{3}\partial_x^2)^{-1} f = \frac{1}{2}\sqrt{\frac{3}{\mu}} e^{-\sqrt{\frac{3}{\mu}}|\cdot|} \ast f =K_{\mu}\ast f$ and $ \partial_x^2(1-\frac{\mu}{3}\partial_x^2)^{-1} f=-\frac{3}{\mu}f+\frac{3}{\mu}(1-\frac{\mu}{3}\partial_x^2)^{-1} f $. Differentiating the second equation of \eqref{BPW} with respect to $ x$ we thus obtain 
\begin{eqnarray}
u_{tx} &= & \frac{3}{\mu}\zeta-  \frac{3}{\mu} \int_{\mathbb{R}} K_{\mu}(\cdot-z)(\zeta-\Frac{\beta}{\eps} b)(z)\, dz     +\frac{3\eps}{\mu}\frac{u^2}{2} - \frac{3\eps}{2\mu} \int_{\mathbb{R}} K_{\mu}(\cdot-z)u^2 (z)dz \nonumber \\
& & \hspace*{2cm} -\Frac{3\beta}{\mu\eps}\int_{\mathbb{R}} K_{\mu}(\cdot-z) b(z)\,dz\nonumber \\
& =& \frac{3}{\mu}\zeta+f_1+f_2+f_3 +f_4\; .\label{ux}
\end{eqnarray}
We would like to estimate the $ L^\infty $ and the $ L^2$-norms of the $f_i $.
The terms with $u $  in the above right-hand side can be easily estimate in the following way 
$$
\Bigl|f_2+f_3 \Bigr|_{L^\infty}
 \le\frac{3}{\mu} |u|_{L^\infty}^2\left(\frac{1}{2} + \frac{1}{2} |K_{\mu}(\cdot)|_{L^1}\right) \le \frac{3}{\mu}|u|_{H^1}^2 \; 
$$
and 
$$
\begin{array}{ll}
\Bigl|f_2+f_3 \Bigr|_{L^2} &\le \frac{3}{2\mu}|u|_{L^4}^2+\frac{3}{2\mu}|K_{\mu}(\cdot) \ast u^2|_{L^2}\le \frac{3}{2\mu}\left(|u|_{L^4}^2+|K_{\mu}|_{L^1} |u^2|_{L^2}\right)\\ 
& \lesssim  \frac{3}{\mu}|u|_{L^4}^2\lesssim \frac{3}{\mu} |u|_{H^1}^2 \; .
\end{array}
$$
We have used that $|K_{\mu}|_{L^1}=1.$ 
To estimate $f_1$ we will make use of \eqref{estA}.
Denoting by $A(t) $ the measurable set of $ \R $ defined by 
$$
A(t)=\{z\in \R \, ,/\,\eps \zeta(t,z)-\beta b(z) \ge M\}\; ,
$$
with $M>1 $ satisfying \eqref{estA}, 
Young's convolution estimates, \eqref{estA} and the fact that $|K_{\mu}|_{L^{\infty}}=\frac{1}{2}\sqrt{\frac{3}{\mu}}$  lead to 
\begin{eqnarray*}
 |K_{\mu}\ast (\zeta-\Frac{\beta}{\eps} b)  |_{L^\infty} & \le & \Bigl|  K_{\mu}\ast ((\zeta-\Frac{\beta}{\eps} b) \chi_{A^\complement})\Bigr|_{L^\infty} +
 \Bigl| K_{\mu}\ast ((\zeta-\Frac{\beta}{\eps} b) \chi_{A})\Bigr|_{L^\infty}\\
& \le  &  |K_{\mu}|_{L^1} |(\zeta-\Frac{\beta}{\eps} b) \chi_{A^\complement}|_{L^\infty} + |K_{\mu}|_{L^\infty} |(\zeta-\Frac{\beta}{\eps} b) \chi_{ A}|_{L^1} \\
 & \le & \frac{1}{\eps}M + \frac{1}{2\eps}\sqrt{\frac{3}{\mu}}|\zeta|_{\Lambda_{\sigma_0}}
\end{eqnarray*}
 and therefore  $$ |f_1  |_{L^\infty} \lesssim \frac{1}{\eps\mu}M + \frac{1}{\eps\mu\sqrt{\mu}}|\zeta|_{\Lambda_{\sigma_0}}.$$
 In the same way, since $|K_{\mu}|_{L^2}=\frac{1}{2}(\frac{3}{\mu})^{\frac{1}{4}}$,
\begin{eqnarray}\label{esta1}
\begin{array}{lcl}
 |K_{\mu}\ast (\zeta-\Frac{\beta}{\eps} b) |_{L^2} & \le  &  |K_{\mu}|_{L^1} |(\zeta-\frac{\beta}{\eps} b) \chi_{A^\complement}|_{L^2} + |K_{\mu}|_{L^2} |(\zeta-\frac{\beta}{\eps} b) \chi_{ A}|_{L^1} \\
 & \le & \frac{1}{C^M_1\eps} |\zeta|_{\Lambda_{\sigma_0}}+ \frac{1}{2\eps}(\frac{3}{\mu})^{\frac{1}{4}} |\zeta|_{\Lambda_{\sigma_0}}\lesssim \frac{1}{\eps\sqrt{\mu}}|\zeta|_{\Lambda_{\sigma_0}}\end{array}
\end{eqnarray} 
and therefore  $$ |f_1  |_{L^2} \lesssim \frac{1}{\eps\mu\sqrt{\mu}}|\zeta|_{\Lambda_{\sigma_0}}.$$
 The last term in (\ref{ux}) is estimate by,
 $|K_{\mu}\ast  b)  |_{L^\infty}  \le | b|_{L^{\infty}}$ and $|K_{\mu}\ast  b)  |_{L^2} \le  | b|_{L^2}$
 And then $$ |f_4  |_{L^\infty} + |f_4  |_{L^2} \leq  \frac{3\beta}{\eps\mu}\Bigl( | b|_{L^{\infty}}+| b|_{L^2}\Bigr)$$
Integrating \eqref{ux} on $ [0,t] $ we get 
\begin{equation}\label{uh}
u_x(t) = u_{0,x} +\frac{3}{\mu} \int_0^t \zeta(s) \, ds + F 
\end{equation}
where, according to  the above estimates and   Proposition \ref{thm1}, 
$$
\begin{array}{lcl}
|F(t)|_{L^\infty}+|F(t)|_{L^2} &  \lesssim&  t \Bigl( \frac{1}{\eps\mu}+  \frac{\beta}{\eps\mu} |b|_{H^1}\Bigr) + \\
& & (e^t-1) \Bigl(\frac{1}{\mu^2} |u_0|_{H^{1}}^2+ \Frac{1}{\mu\sqrt{\mu}} |\zeta_0|_{\Lambda_{\sigma_0}}+(\frac{\beta}{2\eps\sqrt{\mu}})^2|b_x|_{L^2}^2\Bigr),\\
&\leq & C(\mu,\eps,\beta)(e^t-1) \Bigl(1+ |u_0|_{H^{1}}^2+\\ 
& & \hspace*{4.7cm} |\zeta_0|_{\Lambda_{\sigma_0}}+|b|_{H^1}^2\Bigr),\\
& & \hfill \forall t\in [0,T[  \; .
\end{array}
$$
Making use of Holder's inequality, this enables to bound  the second  term of the right-hand side member to \eqref{LM} in the following way :
\begin{align}
\Bigl| -\int_0^t & \int_{\R} \zeta^N(s) ((1-\beta b)u)_x(s) \, ds \Bigr| =\Bigl| -\int_0^t  \int_{\R} \zeta^N(s) (1-\beta b)u_x(s) \, ds \nonumber\\ & \hspace*{6cm} -\int_0^t  \int_{\R} \zeta^N(s) (1-\beta b)_x u(s) \, ds  \Bigr|  \nonumber\\
& = \Bigl|  -\int_{\R}\int_0^t  \zeta^N(s) (1-\beta b)(u_{0,x}+F) -  \int_0^t \int_{\R} \zeta^N(s) (1-\beta b)\int_0^s \zeta(\tau) \, d\tau \, ds\nonumber\\ & \hspace*{6cm} -\int_0^t  \int_{\R} \zeta^N(s) (1-\beta b)_x u(s) \, ds \Bigr| \nonumber \\
& \le  (1+\beta|b|_{W^{1,\infty}})   \Bigl( |u_{0,x}|_{L^{N+1}}+ |F|_{L^{N+1}}+|u|_{L^{N+1}}\Bigr)\int_0^t  |\zeta|_{L^{N+1}}^N(s) \, ds \nonumber \\
& + (1+\beta|b|_{L^\infty})\int_{\R} \int_0^t |\zeta(s)|^N \, ds \int_0^t |\zeta(s)| \, ds\nonumber \\
& \lesssim(1+\beta|b|_{W^{1,\infty}})  \Bigl( |u_{0,x}|_{L^{N+1}}+ |F|_{L^{2}}+ |F|_{L^{\infty}}+|u|_{L^{N+1}}\Bigr) \int_0^t (1+|\zeta(s)|_{L^{N+1}}^{N+1})\, ds \nonumber\\
& \hspace*{4cm}+t(1+\beta|b|_{L^\infty}) \int_0^t \int_{\R} |\zeta(s)|^{N+1} \, ds\nonumber \\
& \leq C(\mu,\eps,\beta) (e^t-1)(1+|b|_{W^{1,\infty}})  \Bigl( 1+|b|_{H^1}^2+ |u_0|_{H^{\frac{3}{2}+}}^2+ |\zeta_0|_{\Lambda_{\sigma_0}}\Bigr)\nonumber\\
& \hspace*{6cm}\times \Bigl(1+\int_0^t |\zeta|_{L^{N+1}}^{N+1}
(s) \, ds\Bigr)
\label{z1}
\end{align}
where in the penultimate  step we perform Holder's inequalities in time.

Finally, since $ \frac{3}{\mu}\zeta\ge -\frac{3}{\eps\mu} (1-\beta b) $ on $[0,t] $, \eqref{uh} leads to 
$$
u_x(t) \ge  u_{0,x} - \frac{3}{\eps\mu}(1-\beta b)t  + F \; .
$$
Since $ N+1 $ is even, this enables to control the last term of the right-hand side member to \eqref{LM} in the following way :
\begin{eqnarray}
- \frac{N}{N+1}\int_0^t \int_{\R}  \zeta^{N+1} u_x & \le & \Bigl(t+|u_{0,x}|_{L^\infty}+|F|_{L^\infty(]0,t[\times \R)}\Bigr)
 \int_0^t \int_{\R}\zeta^{N+1} 
\nonumber \\
& \leq & C(\mu,\eps,\beta)\,e^t  \Bigl( 1+ |u_0|_{H^{\frac{3}{2}+}}^2+ |\zeta_0|_{\Lambda_{\sigma_0}}+|b|_{H^1}^2\Bigr)\nonumber \\
& & \hspace*{4cm}\times  \int_0^t |\zeta|_{L^{N+1}}^{N+1}\; .\label{z2}
\end{eqnarray}
Gathering \eqref{LM} and \eqref{z1}-\eqref{z2}, we infer that $ \gamma(t) =(\int_0^t |\zeta(s)|_{L^{N+1}}^{N+1} ds)^{\frac{1}{N+1}} $
 satisfies the following differential inequality on $]0,T[ $
 \begin{equation}
 \begin{array}{r}
\frac{d}{dt}  \gamma^{N+1}(t)\lesssim {|\zeta_0|_{L^{N+1}}^{N+1}}+ e^t (N+1)(1+|b|_{W^{1,\infty}}) \Bigl( 1+ |u_0|_{H^{\frac{3}{2}+}}^2+ |\zeta_0|_{\Lambda_{\sigma_0}}+|b|_{H^1}^2\Bigr)\\
 \times\Bigl(1+\gamma^{N+1}(t)\Bigr)\; .
\end{array} 
\label{et}
 \end{equation}
Making use of  Sobolev inequalities and \eqref{boundd},  Gronwall's inequality  ensures that there exists $\tilde{C}_T>0 $ only depending on $ T $, $|u_0|_{H^{\frac{3}{2}+}},\,  |\zeta_0|_{H^{\frac{1}{2}+}} $ and $|b|_{W^{1,\infty}\cap H^1}$ such that for all $t\in [0,T] $, 
$$
\gamma^{N+1}(t) \lesssim \exp\Bigl((N+1)\tilde{C}_T\Bigr)(1+|\zeta_0|_{L^{N+1}}^{N+1})  \; . $$
Then re-injecting this estimate in \eqref{et} we obtain 
$$
\sup_{t\in [0,T[} |\zeta(t)|_{L^{N+1}}^{N+1} \lesssim \tilde{\tilde{C}}_T (N+1)  \exp\Bigl((N+1)\tilde{C}_T\Bigr)(1+|\zeta_0|_{L^{N+1}}^{N+1})
$$
which leads to 
$$
\sup_{t\in [0,T[} |\zeta(t)|_{L^{N+1}}\le C_T \; ,$$
where $C_T $ and $\tilde{\tilde{C}}_T>0 $ only depend on $ T $, $|u_0|_{H^{\frac{3}{2}+}},\, |\zeta_0|_{H^{\frac{1}{2}+}} $ and $|b|_{W^{1,\infty}\cap H^1}$. 
Letting $ N\to +\infty $ this proves the  estimate on the first term in  \eqref{estest}.
Finally, the estimate on the second term in \eqref{estest} follows directly from the first one together with \eqref{uh}.\\

The continuity of the flow-map follows directly from Proposition \ref{essentiel}.
\end{proof}
Finally we can state the following theorem as a consequence of the previous results together with Lemma \ref{lemminmax} in the appendix.
\begin{theorem}\label{thmreg}
Let $s>1/2$ and $b\in  H^{s+1}$. For $(\zeta_0,u_0) \in H^s\times H^{s+1}_\mu$ such that $1+\eps\zeta_0 -\beta b>0$,  the  Boussinesq system (\ref{BPW}) has a  solution $(\zeta,u)$ in $C(\R_+,H^s\times H^{s+1}_\mu )\cap C^1(\R_+,H^{s-1}\times H^{s}_\mu ) $. This  solution satisfies $ 1+\eps \zeta(t)-\beta b >0 $,  for any $t\ge 0 $, and is the unique solution of \eqref{BPW} that belongs to $ L^\infty_{loc}(\R_+; H^s\times H^{s+1}_\mu )$.\\
For any $ T>0 $, the  flow-map $S :(\zeta_0,u_0) \longrightarrow (\zeta,u)$ is  continuous  
 from $H^s\times H^{s+1}_\mu  $ into $ C([0,T]; H^s\times H^{s+1}_\mu) $.
\end{theorem} 
\section{Global entropy solutions  of the Boussinesq system}
In this section, we study the existence of weak solutions for the Boussinesq system (\ref{BPW}) for initial condition  $(\zeta_0,u_0) \in  \Lambda_{\sigma_0}\times H^1_\mu$ with $ 1+\eps\zeta_0 +\beta b>0 $ a.e. on $ \R $. To do so,we suppose $b\in  H^1,$ we regularize the initial data by a mollifiers sequence $ (\rho_n)_n\subset D(\R)$ by setting  $b_n=\rho_n*b,\, \zeta_{0,n}=\rho_n* \zeta_0 $ and $u_{0,n}=\rho_n* u_0 $, where $\rho_n(\cdot)=n\rho(n\cdot), \; \rho\in D(\R)$ such that $$0\leq \rho\leq 1, \mbox{ supp } \rho \subset [0,1] \mbox{ and } \displaystyle \int_\R \rho \, dx=1 .$$  
Note that $(u_{0,n})_n\subset H^s$  and it is  bounded in $H^1$ with $\|u_{0,n}\|_{H^1_\mu} \leq \|u_0\|_{H^1_\mu}$
 and $ u_{0,n} \to u_0 $ in $H^1(\R) $. For $\zeta_{0,n}$, we first notice that $ \zeta_0\in   \Lambda_{\sigma_0} $ and $ 1+\eps\zeta_0-\beta b>0 $ a.e. on $ \R $. Since $ 1+\eps\zeta_{0,n}-\beta b_n= \rho_n * (1+\eps\zeta_0-\beta b) $, it follows that $1+\eps\zeta_{0,n}-b_n>0 $ on $\R $. Moreover, using  (\ref{estA}) and proceeding exactly as (\ref{esta1}) by replacing  $\frac{1}{2}\sqrt{\frac{3}{\mu}} e^{-\sqrt{\frac{3}{\mu}}|\cdot|}$ by $\rho_n$ it is straightforward to check that  $\zeta_{0,n} \in L^2$ with $|\zeta_{0,n}|_{L^2}\leq \frac{c_n}{\eps}|\zeta_0|_{ \Lambda_{\sigma_0}}+\frac{\beta}{\eps}|b|_{L^2}$ where $c_n$  depends on $\|\rho_n\|_{L^2}$. Similary, we can verify that $\zeta_{0,n} \in H^s$, for $s\geq 0$.\\
Now,  consider $(\zeta_n, u_n)$ the solution of (\ref{BPW}) emanating from $(\zeta_{0,n},u_{0,n})$ given by Proposition \ref{prop1}. we will prove that $(\zeta_n,u_n)$ has a subsequence which converges to a weak solution of the Boussinesq system (\ref{BPW}) with initial data $(\zeta_0,u_0)$. Note that $(\zeta_n,u_n )$ satisfies the entropy estimate (\ref{inq1}) which implies

$$
\begin{array}{l}
 \frac{1}{2} |u_n(t)|_{H_{\mu}^1}^2 + 
 \frac{1}{\eps^2} \int_{-\infty}^{+\infty}  \sigma_0(1+\eps\zeta_n(t,x)-\beta b_n(x))dx  \le  
 \\
\hspace{1cm}  \left( |u_{0,n}|_{H_{\mu}^1}^2+\frac{1}{\eps^2} \int_{-\infty}^{+\infty}  \sigma_0(1+\eps\zeta_{0,n}(x)-\beta b_n(x))dx + (\frac{\beta}{2\eps})^2 |b_{n,x}|_{{\rm L}^2}^2 \right)e^t, \\
\hfill \forall t\in [0,T_0].
\end{array}
$$

The  $ H^1$-convergence of $ (u_{0,n}) $ towards $ u_0 $ ensures that the first term of the above right-hand side converges to $ \|u_0\|_{H^1}^2 $. For the second term one has to work a little more. We follow  \cite{Sch} and use the convexity of $ \sigma_0$ and Jensen inequality to get that 
$$
\sigma_0(1+\eps\zeta_{0,n}-\beta b_{0,n})=\sigma_0 \Bigl( \int_{\R} (1+\eps\zeta_0-\beta b)\rho_n (\cdot-z) dz\Bigr) 
\le \int_{\R} \sigma_0(1+\eps\zeta_0-\beta b) \rho_n (\cdot-z) dz \; .
$$
Therefore, integrating on $ \R $, using Fubini and $\int_{\R} \rho_n =1 $, we obtain 
$$
 \int_{-\infty}^{+\infty}  \sigma_0(1+\eps\zeta_{0,n}-\beta b_{0,n})dx \le
  \int_{-\infty}^{+\infty}  \sigma_0(1+\eps\zeta_{0}-\beta b)dx \; . 
$$
We thus are lead to the following uniform estimate on $ \R_+$ : 
\begin{equation} \label{inq2}
\begin{array}{l}
 \frac{1}{2} |u_n(t)|_{H_{\mu}^1}^2 + 
 \frac{1}{\eps^2} \int_{-\infty}^{+\infty}  \sigma_0(1+\eps\zeta_n(t,x)-\beta b_n(x))dx  \le  
 \\
\hspace{1cm}  \left( |u_0|_{H_{\mu}^1}^2+\frac{1}{\eps^2} \int_{-\infty}^{+\infty}  \sigma_0(1+\eps\zeta_0(x)-\beta b(x))dx + (\frac{\beta}{2\eps})^2 |b_{x}|_{{\rm L}^2}^2 \right)e^t, \\
\hfill \forall t\in [0,T_0].
\end{array} .
\end{equation}
 In the sequel we will make a constant use of the following lemma that can be easily deduced from \eqref{estA} and \eqref{inq2}.
\begin{lemma}\label{cor2}
let $\chi \in L^1(\R)\cap L^\infty(\R)$. Then
\begin{eqnarray}
\int_\R \left | \zeta_n (t,x)\chi(x)\right|dx \leq c,
\end{eqnarray}
for all $t$, where $c$ is independent of $n$. 
\end{lemma}
 \begin{proof} Let $M$ be defined as in (\ref{estA}). Then by \eqref{inq2}, we have
 \begin{eqnarray}
 \begin{array}{lcl}
\displaystyle \int_\R|\zeta_n(t,x)| |\chi(x)|dx&= &\frac{1}{\eps}\displaystyle \int_{\{x/{-1<\eps\zeta_n(t,x)-\beta b_n(x)\leq M}\}} |\eps \zeta_n(t,x)-\beta b_n(x)||\chi(x)|dx\\
&& +\frac{1}{\eps}\displaystyle \int_{\{x/{\eps\zeta_n(t,x)-\beta b(x) \geq M}\}}|\eps \zeta_n(t,x)-\beta b_n(x)||\chi(x)|dx \\
&&\quad + \frac{\beta}{\eps}\displaystyle\int_\R |b_n(x)|\chi(x)| dx \\
&\leq &\Bigl( \frac{M}{\eps}+\frac{\beta}{\eps}|b_n|_{L^\infty}\Bigr) \displaystyle\int_\R |\chi(x)|dx \\
&& +\frac{1}{\eps}|\chi|_{L^\infty} \int_\R \sigma_0(1+\eps\zeta_n(t,x)-\beta b_n)dx
 \leq\quad c.
\end{array}
\end{eqnarray}
\end{proof}

\begin{proposition}\label{T6} Consider the sequence $(\zeta_n,u_n) $ constructed above for $(\zeta_0,u_0)\in  \Lambda_{\sigma_0}\times H^1_\mu$ { (in particular $1+\eps\zeta_0-\beta b >0$ {a.e. on $ \R $)}}. Then there exists a subsequence $\left((\zeta_{n_k},u_{n_k})\right)_k$ and $(\zeta,u) \in L^1_{loc}(]0,+\infty[\times \R)\times (L^\infty_{loc}(]0,+\infty[,H^1_\mu(\R))\cap C(\R^+\times\R))$ such that $(\zeta_{n_k})_k$ converges weakly to $\zeta$ in $L^1$ on every compact of $]0,+\infty[ \times \R$ and $(u_{n_k})_k$ converges  weakly-$\ast$ in $L^\infty_{loc}(]0,+\infty[,L^\infty (\R))$ and strongly, for any $ T>0$ , in   $C([0,T],C(\R))$ (and then in $C([0,T],L^2_{loc}(\R)))$ to $u$.
\end{proposition} 
{\bf{Proof.}} 
 First, applying the above lemma with $ \chi=1\!\!1_{[-A,A]} $ for $ A>0 $ we obtain that $ (\zeta_n(t))_n $ is bounded in $ L^1_{loc}(\R) $ uniformly in $ t\in\R_+ $. In particular, $(\zeta_n)_n $ is bounded in $L^1_{loc}(\R_+\times\R) $. According to Dunford-Pettis Theorem (see  \cite{DunSchw} Vol I p.294),  to prove that $(\zeta_n)_n $ is weakly compact in $ L^1(]0,T[\times ]-A,A[)$, it suffices to check that for any $ \nu
 >0 $ there exists $ \delta>0 $ such that for any bounded measurable set 
$B\subset]0,T[\times ]-A,A[ $ with $ |B|<\delta $ it holds 
$$
\sup_{n\in\N} \int_{B} |\zeta_n|(t,x) dx\, dt < \nu \; .
$$
But this follows directly from \eqref{estA} and \eqref{inq2}. Indeed, proceeding as in the proof of the above lemma, we easily check that for any $ k\ge M $ and any $B\subset]0,T[\times ]-A,A[ $,
\begin{align*}
\int_{B} |\zeta_n|(t,x) dx\, dt & \le \frac{1}{\eps} \int_{B\cap (-1<\eps\zeta_n-\beta b_n\le k)}|\eps\zeta_n-\beta b_n| dx\, dt\\
&\quad +\frac{1}{\eps}\int_{B\cap (\eps\zeta_n-\beta b_n>k)}|\eps\zeta_n-\beta b_n| dx\, dt +\frac{\beta}{\eps}\int_{B}|b_n|dx\,dt\\
& \le \frac{1}{\eps}(k+\beta|b|_{\infty}) |B| + \frac{1}{\eps}(C_2^M \ln k)^{-1} \int_0^T \int_{\R} \sigma_0(1+\eps\zeta_n-\beta b_n) dx\, dt \\
& \le \frac{1}{\eps}(k+\beta|b|_{\infty}) |B| +\frac{T}{\eps}  |\zeta_0|_{\Lambda_{\sigma_0}} (C_2^M \ln k)^{-1}\; ,
\end{align*} 
that clearly gives the desired result by taking $ k $ large enough. 

 Now, let us tackle the strong convergence of $ (u_n)_n $. By (\ref{inq2}), we have that $(u_n)_n$ is bounded in $L^\infty([0,T], H^1_\mu(\mathbb R))$. Then, $(u_n)_n$ is bounded in $L^\infty(]0,T[,L^\infty (\R)).$  We deduce that it has a weakly-$\ast$  convergent subsequence in $L^\infty(]0,T[,L^\infty (\R))$. Now,  we will prove that $(\partial_t u_n)_n$ is bounded in $L^\infty(]0,T[, L^2(\mathbb R))$ and after we use a theorem of Aubin-Simon to get the strong convergence. To do so, we recall that by the equation it holds 
$$ \partial_t u_n(t,x)=\int_\mathbb R k_\mu(x-z)(\zeta_n+\eps u_n^2/2)(t,z)dz.$$
where $k_\mu(x):=\partial_x K_\mu(x)=\displaystyle \frac{3}{2\mu}\sgn(x)e^{-\sqrt{\frac{3}{\mu}}|x|}$.
Now by (\ref{inq2}) and  Young's convolution estimates, we have
$$\int_\mathbb R k_\mu(x-z) u_n^2/2(t,z)dz=|k* u_n^2/2|^2_{L^2}\leq c |k_\mu|_{L^1} |u_n^2|_{L^2}\leq C(\mu). $$

For the first integral on $\zeta_n$, we will use, as in \cite{Ami}, the fact that $|\sigma_0(1+\zeta_n)|_{L^1} \leq cte$ given by (\ref{inq2}) and the property of the mapping $\sigma_0(1+\cdot)$ giving in (\ref{estA}). Again denoting by $A_n(t)=\{ x\in \R/ \eps\zeta_n(t,x)-\beta b(x)\geq M\}$,  with $ M \ge 1 $ as in \eqref{estA},

 we write
\begin{align*}
\int_\mathbb R k(x-z)\zeta_n(t,z)dz& =\frac{1}{\eps}\int_{A_n^c} k(x-z)(\eps\zeta_n-\beta b_n)(t,z)dz\\
 & +\frac{1}{\eps}\int_{A_n} k(x-z)(\eps\zeta_n-\beta b_n)(t,z)dz+ \frac{\beta}{\eps}\int_{\R}k(x-z)b_n(z)\,dz\\
& =\frac{1}{\eps}(f_1+f_2+f_3).
\end{align*}
 By  Young's convolution estimates, (\ref{estA}) and (\ref{inq2}), we get 
\begin{align*}
    |f_1(t,\cdot)|_{L^2}\leq |k|_{L^1} | \zeta_n 1_{A_n^c}|_{L^2}&\leq  (C^M_1)^{-1/2} \left(\int_{A_n^c} \sigma_0(1+\eps\zeta_n(t,x)-\beta b(x))dx\right)^{1/2}\\
    &\lesssim |\zeta_0|_{\Lambda_{\sigma_0}}^{1/2},
\end{align*}

$$ |f_2(t,\cdot)|_{L^\infty}\leq \displaystyle \frac{1}{2} \int_{A_{n}}|\eps\zeta_n (t,x)-\beta b(x)|dx \lesssim |\zeta_0|_{\Lambda_{\sigma_0}}$$

$$  |f_2(t,\cdot)|_{L^1}\leq |k|_{L^1} | \zeta_n 1_{A_n}|_{L^1} \lesssim |\zeta_0|_{\Lambda_{\sigma_0}}.$$

Then, we obtain
$$  |f_2(t,\cdot)|_{L^2}\leq  |f_2|_{L^\infty}^{1/2} |f_2|_{L^1}^{1/2}\leq cte,$$
and
$$|f_3(t,\cdot)|_{L^2}\leq \beta |b|_{L^\infty}^{1/2} |k|_{L^1}^{1/2}.$$
Combining the above estimates, we deduce that $\|\partial_t u_{n}\|_{L^2}$ is  bounded uniformly in $n$ and $t$. \\
Next, we prove that $(u_n)_n$ has a strongly convergent subsequence  in \\ $C([0,T]; C(\R))$, i.e. in $C([0,T]; C(K))$ for every compact $K$ of $\R$. For this end, set $K_m=[-m,m]$, by compact Sobolev injection and Aubin-Simon theorem (see \cite{S87}, Corollary 4 p. 85), we have that $ E^m_{\infty,\infty}$ is compactly embedded in $ C([0,T], C(K_m))$  where $$ E^m_{\infty,\infty}= \{ u\in L^\infty(]0,T[,H^1(K_m)) \mbox{ such that }\displaystyle \partial_t u \in L^\infty(]0,T[,L^2(K_m))\}.$$ By the preceding calculations, we have proved that $(u_n)_n$ is bounded in $E^m_{\infty\infty}$. We deduce that it has a  subsequence $(u_{n_k}^m)_k$ strongly convergent in $C([0,T],C(K_m))$ $(\mbox{also in }  C([0,T],L^2(K_m)))$. By applying the diagonal extraction processus, we can construct a subsequence $(u_{n_k})_k$ which is strongly convergent in $C([0,T],C(K_m))$, for every $m\geq 1 $ and thus  in $C([0,T],C(K))$ $(\mbox{ and then in }  C([0,T],L^2(K)))$ for every compact $K$ of $\R$.
This completes the proof of the theorem.
In the next theorem we construct  a weak solution for the Boussinesq system.
\begin{theorem}\label{T72}
Let $(\zeta_n,u_n)$ be as in Proposition \ref{T6}. Then, the limit functions $(\zeta,u)$ obtained by Proposition \ref{T6} is a weak solution for the Boussinesq system with initial data $(\zeta_0,u_0)$ verifying that $(\zeta,u) \in L^\infty(\R^+; \Lambda_{\sigma_0}\times H^1_\mu).$
\end{theorem}

\begin{proof} Let $\varphi\in C_c^\infty(]0,+\infty[\times \R)$. Multiplying (\ref{BPW}) by $\varphi$ and integrating, we obtain 

\begin{eqnarray}\label{iT7}
\int_0^{+\infty}\int_\R\zeta_n \varphi_t dx dt+\int_0^{+\infty}\int_\R(u_n+u_n(\eps\zeta_n-\beta b_n )\varphi_x dx dt=0
\end{eqnarray}
and
\begin{eqnarray}\label{iT8}
\int_0^{+\infty}\int_\R u_n \varphi_t dx dt &+\int_0^{+\infty}\int_\R(\eps u_n^2/2+\zeta_n)\varphi_x dx dt - \int_0^{+\infty}\int_\R u_n \varphi_{xxt} dx dt \nonumber\\
&\hspace{1cm }=0.
\end{eqnarray}
By taking the limit when $n$ tends to infinity, we have to prove that
$$ \int_0^{+\infty}\int_\R\zeta \varphi_t dx dt+\int_0^{+\infty}\int_\R(u+u(\eps\zeta -\beta b)\varphi_x dx dt=0$$
and 
$$\int_0^{+\infty}\int_\R u \varphi_t dx dt+\int_0^{+\infty}\int_\R(\eps u^2/2+\zeta)\varphi_x dx dt- \int_0^{+\infty}\int_\R u \varphi_{xxt} dx dt=0 $$
Since $\varphi$ is with compact support in $]0,+\infty[\times \R$ and $(\zeta_n)_n$ converges weakly in $L^1_{loc}$ to $\zeta$ we obtain
$$ \lim_{n\rightarrow +\infty}\int_0^{+\infty}\int_\R\zeta_n \varphi_tdxdt = \int_0^{+\infty}\int_\R\zeta \varphi_tdxdt.$$
And the strong convergence of $u_n$ to $u$ in $C([0,T],L^2(K))$ implies that 
$$ \lim_{n\rightarrow +\infty}\int_0^{+\infty} \int_\R (u_n-u) \varphi_x dx dt =0.$$
Let $S$ be the support of $\varphi$ and suppose that $ S \subset ]0,T[ \times ]c,d[ \subset ]0,+\infty[\times \R$. Then, we write
$$ \int_0^{+\infty}\int_\R (\zeta_n u_n - \zeta u)\varphi_xdxdt= \int_S \zeta_n(u_n-u)\varphi_x dx dt +\int_Su(\zeta_n-\zeta)\varphi_x dx dt $$
 Since  $u_n\to  u $ in  $ C([0,T];C(K)), $ for every $K$ compact of $\R$ 
and  $ (\zeta_n)_n $ is bounded in $ L^1_{loc} $ we deduce that the first term in the above right-hand side member tends to $0 $ as $n\to +\infty $. Notice that the limit for second term  follows directly from the weak convergence of $(\zeta_n)_n$ in $ L^1_{loc} $ and the fact that $u\varphi_x \in L^\infty(S)$. In the same way we write
$$ \int_0^{+\infty}\int_\R (b_n u_n - b u)\varphi_xdxdt= \int_S b_n(u_n-u)\varphi_x dx dt +\int_Su(b_n-b)\varphi_x dx dt, $$  again using the convergence of $u_n$ to $u$ in $C([0,T],L^2(K))$ and the one of $b_n$ to $b$ in $L^2(K))$ we deduce that 
$$\int_0^{+\infty}\int_\R (b_n u_n - b u)\varphi_xdxdt\longrightarrow 0.$$
We finally obtain
$$\int_0^{+\infty}\int_\R\zeta \varphi_t dx dt+\int_0^{+\infty}\int_\R(u+u(\eps\zeta-\beta b) )\varphi_x dx dt=0,$$
which implies that $(\zeta,u)$ satisfies the first equation of (\ref{BPW}) in the distribution sense.\\ 
For the second equation (\ref{iT8}), the proof is  direct  using the weak convergence of $(\zeta_n)$ and the strong convergence of $(u_n)$. \\ \\
It is still  to prove that the limit $(\zeta,u)$ satisfies the initial data $(\zeta_0,u_0).$  Recall that $(u_n) $ converges to $u$ in $C([0,T],C(K))$ (i.e. in $C([0,T]\times K))$)  for every compact $K$ of $\R$ and that  $u_{0,n}$ converges to $u_0$ in $H^1(\R)$. This enough to implies that $u(0,x)=u_0(x)$ for a.e. $x\in \R$. 
In fact we notice that 
\begin{eqnarray*}
\begin{array}{ll}
\displaystyle || (u(t,\cdot)-u_0(\cdot))||_{\infty,K} \leq   \displaystyle || (u(t,\cdot)-u_n(t,\cdot))||_{\infty,K}  & \\ 
\hspace*{3cm}+ \displaystyle || (u_n(t,\cdot)-u_{0,n}(\cdot))||_{\infty,K} +\displaystyle || (u_{0,n}(\cdot)-u_0(\cdot)))||_{\infty,K}  .&
\end{array}
\end{eqnarray*}
The above convergence results force the first and the third  term of the above right-hand side to converge towards $ 0 $ uniformly in $ t\in [0,T] $ whereas the continuity of $ u_n $ force the second term to tends to $0$ as $ t\searrow 0 $ for each fixed $n\in \N $, 
 and thus  $u(0,x)=u_0(x)$ for a.e. $x\in \R$. \\
Let us now prove that $\zeta$ satisfies the initial condition. For $\varphi \in C^\infty_c(\R)$ and $(t_k)_k \in [0,T]$ converging to $0$, we have
\begin{eqnarray}\label{ic}
\begin{array}{ll}\displaystyle \left|\int_\R (\zeta(t_k,x)-\zeta_0(x))\varphi dx\right| \leq   \displaystyle\left|\int_\R (\zeta(t_k,x)-\zeta_n(t_k,x))\varphi dx\right| &\\ 
\hspace*{2cm}+ \displaystyle \left|\int_\R (\zeta_n(t_k,x)-\zeta_{0,n}(x))\varphi dx\right|+ \left|\int_\R (\zeta_{0,n}(x)-\zeta_0(x))\varphi dx\right|. &
\end{array}
\end{eqnarray}
For the first integral, we proceed as in \cite{Sch}, Claim $4.2$  of Theorem $4.2$]. So by Dunford's lemma and  Lemma \ref{cor2}, for each $t$ there exists a subsequence $(\zeta_n^t)_n$ of $(\zeta_n)_n$ such that 
$$ \lim_{n\rightarrow +\infty} \int_\R (\zeta(t,x)-\zeta^t_n(t,x))\varphi dx =0.$$
Now applying the diagonalization process to $ (\zeta_n^{t_k})_{n,k}$, we can extract a subsequence $(\zeta_k)_k$ such that
$$\lim_{k\rightarrow +\infty} \int_\R (\zeta(t_k,x)-\zeta_k(t_k,x))\varphi dx =0.  $$ 
For the second integral, using the integral representation of $\zeta_k$ given by \eqref{BPW}$_1$, we get
\begin{eqnarray} \begin{array}{l} \displaystyle \int_\R (\zeta_k(t_k,x)- \zeta_{0,k}(t,x)) \varphi (x) dx \\
\qquad = -\displaystyle\int_\R \int_0^{t_k} \left((u_k(s,x) +u_k (\eps\zeta_k(s,x)-\beta b_k(x))\right)_x \varphi(x) dsdx \\
\qquad =\displaystyle \int_0^{t_k} \int_\R \left((u_k(s,x) +u_k (\eps\zeta_k(s,x)-\beta b_k(x))\right) \varphi_x(x)dx ds .\end{array} 
\end{eqnarray}
By Lemma \ref{cor2}  it holds 
$$
\begin{array}{ll}
 \left| \displaystyle \int_0^{t_k} \int_\R \left((u_k(s,x) +u_k (\eps\zeta_k(s,x)-\beta b_k(x))\right)\varphi_x(x)dx ds\right| &\\ 
 \hspace*{2cm} \leq |u_k|_{L^\infty} \displaystyle \int_0^{t_k} \int_\R \left (| \varphi_x(x)| + 
|\zeta_k(s,x)\varphi_x(x)|+|b_k(x)\varphi_x(x)|\right) dx & \\
\hspace*{2cm} \leq c_1 t_k, \\
\end{array}
$$
where $c_1$ is independent of $k$. So for the subsequence $ \zeta_k$, we obtain that 
$$  \displaystyle \left|\int_\R (\zeta_k(t_k,x)-\zeta_{0,k}(x))\varphi dx\right|\leq c_1 t_k. $$
The last integral in (\ref{ic}) goes to $0$ since  $ \zeta_0\in L^1_{loc} $ and thus $(\zeta_{0,k})_k $ converges to $\zeta_0$ in $L^1_{loc}$. Then, by using the subsequence $(\zeta_k)_k$ in (\ref{ic}), we deduce that
$$ \lim _{k\rightarrow +\infty} \displaystyle \left|\int_\R (\zeta(t_k,x)-\zeta_0(x))\varphi dx\right|=0. $$ 

Now to finish the proof we can  proceed as  Schonbek in \cite{Sch} (Theorem 5.2) to show that $(\zeta,u) \in L^\infty(\R^+; \Lambda_{\sigma_0}\times H^1_\mu)$.
\end{proof}

\section{\bf{Appendix}}
\subsection{Proof of Proposition \ref{propcom}}
Let $ N>0$.  We follow \cite{Lanlivre13}. By Plancherel and the mean-value theorem, 
 \begin{align*}
\Bigl| ( [P_N, P_{\ll N}f] g_x)(x)\Bigr| &=\Bigl|( [P_N, P_{\ll N}f] \tilde{P}_N g_x)(x)\Bigr| \\
 & =\Bigl|  \int_{\R} {\mathcal F}^{-1}_x(\varphi_N)(x-y) P_{\ll N}f(y)  \tilde{P}_N g_x(y) \, dy \\ 
 &\qquad -
 \int_{\R}  P_{\ll N}f(x)  {\mathcal F}^{-1}_x(\varphi_N)(x-y) \tilde{P}_N g_x(y) \, dy\Bigr| \\
 & = \Bigl|  \int_{\R}(P_{\ll N}f (y) -P_{\ll N}f(x)) N  {\mathcal F}^{-1}_x(\varphi)(N(x-y))  \tilde{P}_N g_x(y) \, dy\Bigr| \\
& \le \|P_{\ll N}f_x\|_{L^\infty_x}\int_{\R} N |x-y| |{\mathcal F}^{-1}_x(\varphi)(N(x-y))| |\tilde{P}_N g_x(y)| \, dy 
  \end{align*}
  Therefore, since $N |\cdot| |{\mathcal F}^{-1}_x(\varphi)(N \cdot)|=|{\mathcal F}^{-1}_x(\varphi')(N \cdot) | $ we deduce from Young's convolution and Bernstein inequalities that 
 \begin{align*} 
 \Bigl| [P_N, P_{\ll N}f] g_x)|_{L^2} & \lesssim  N^{-1}|P_{\ll N}f_x|_{L^\infty_x}  |\tilde{P}_N g_x|_{L^2} 
 \lesssim |P_{\ll N}f_x|_{L^\infty_x}  |\tilde{P}_N g|_{L^2} \; .   \end{align*}
This completes the proof of estimation (\ref{cm1}).
\subsection{Proof of Proposition \ref{propproduit}} Let us prove estimate (\ref{proN1}).  Using Bernstein inequality and discrete Young' inequalities, we have
 $$
 \begin{array}{lcl}
N^s |P_N ( P_{\gtrsim N} f_x\, g)|_{L^2} &\lesssim & N^s |g|_{L^\infty_x}\Bigl(\sum_{K\gtrsim N}  |P_{K} f_x|_{L^2_x}^2\Bigr)^{1/2}\\ 
&\lesssim & |g|_{L^\infty_x} 
 \sum_{K \gtrsim N} (\frac{N}{K})^s   |P_K f_x|_{H^s}  \\ 
 &\lesssim &   \delta_N |g|_{L^\infty_x} |f|_{H^{s+1}}
 
 \end{array}
 $$
Now it remains to prove 
 $$ N^s |P_N ( P_{\gtrsim N} f \, g_x)|_{L^2} \lesssim \delta_N | f |_{H^{s+1}}|g|_{L^\infty}.$$
To do so, we write $P_N ( P_{\gtrsim N} f \, g_x)= \partial_xP_N ( P_{\gtrsim N} f \, g)- P_N ( P_{\gtrsim N} f_x \, g)$. The second term can be treated as above to obtain 
 $$N^s |P_N ( P_{\gtrsim N} f_x \, g)|_{L^2} \lesssim  \delta_N | f |_{H^{s+1}}|g|_{L^\infty}.$$
For the first term, we have
\begin{equation}\nonumber
 \begin{array}{lcl}
  N^s |\partial_xP_N ( P_{\gtrsim N} f \, g)|_{L^2}& \lesssim & N^s N |P_N(P_{\gtrsim N} f \, g)|_{L^2} \\
  & \lesssim & N^s N  | P_N(\sum_{K\gtrsim N} P_K f \, g) |_{L^2_x}\\
   & \lesssim & |  f |_{H^{s+1}} |g |_{L^\infty_x}  \sum_{K\gtrsim N} \delta_K   (\frac{N}{K})^{s+1}\\
 &\leq& \delta_N | f|_{H^{s+1}}|g|_{L^\infty}
 \end{array}
 \end{equation}
 \\
Finally to prove estimate (\ref{proN2}), we first notice that it follows directly from \eqref{proN1} for $ s>3/2 $ since $ H^{s-1}(\R) \hookrightarrow L^\infty(\R) $. For $ 1/2<s\le 3/2 $, we start by noticing that 
$$
P_N(P_{\gtrsim N} f \, g_x) =P_N(P_{\sim N} P_{\lesssim N} g_x)  + P_N(\sum_{K\gtrsim N} P_K f  \tilde{P}_K g_x)\; .
$$
The contribution of the first term of the above  right-hand side  is easily estimated by 
\begin{align*}
N^{s-1} |P_N(P_{\sim N} f  P_{\lesssim N} g_x)|_{L^2} & \lesssim N^{s-1} |P_{\sim N} f |_{L^\infty} N^{2-s} |g|_{H^{s-1}} \\
& \lesssim N  |P_{\sim N} f |_{L^\infty} |g|_{H^{s-1}} \\
& \lesssim \delta_N  |f |_{H^{s+1}} |g|_{H^{s-1}}
\end{align*}
since $ s>1/2 $. 
On the other hand, the contribution of the second term can be estimated by 
\begin{align*}
N^{s-1} \Bigl|P_N(\sum_{K\gtrsim N} P_K f  \tilde{P}_K g_x)\Bigr|_{L^2} & 
\lesssim N^{s-1} N^{1/2}  \Bigl| P_N(\sum_{K\gtrsim N} P_K f  \tilde{P}_K g_x)\Bigr|_{L^1}\\
& \lesssim N^{s-1/2} \sum_{K\gtrsim N} 
K^{-1-s} |P_{\sim K} f |_{H^{s+1}} K^{2-s} |\tilde{P}_K g|_{H^{s-1}} \\
& \lesssim N^{s-1/2} |f |_{H^{s+1}} |g|_{H^{s-1}} \sum_{K\gtrsim N} 
K^{1-2s}\\
& \lesssim N^{1/2-s} |f |_{H^{s+1}}  |g|_{H^{s-1}}
\end{align*}
that is suitable since $s>1/2 $. 
\subsection{Proof of Proposition \ref{propproduit2}}
We write 
\begin{align}
|fg|_{H^s}^2 & \lesssim |P_{\lesssim 1} (fg) |_{L^2_x}^2 +\sum_{N\gg 1} N^{2s} |P_N(fg)|_{L^2_x}^2 \nonumber\\
&\lesssim |P_{\lesssim 1} (P_{\lesssim  1} f \, P_{\lesssim 4} g) |_{L^2_x}^2 +|P_{\lesssim 1} (P_{\gg 1} f \, g) |_{L^2_x}^2\nonumber \\
& +\sum_{N\gg 1} N^{2s} (|P_N(P_{\ll N} f \tilde{P}_N g)|_{L^2_x}^2+|P_N(P_{\gtrsim N} f   g)|_{L^2_x}^2)\label{vv}
\end{align}
For any $ s\in \R $, it is direct to check that 
\begin{align*}
 |P_{\lesssim 1} (P_{\lesssim 1 } f P_{\lesssim 4} g) |_{L^2_x}^2 + & \sum_{N\gg 1} N^{2s} |P_N(P_{\ll N} f \tilde{P}_N g)|_{L^2_x}^2\\
  &  \lesssim |P_{\lesssim 1} f|_{L^\infty_x}^2 |P_{\lesssim 4} g |_{L^2_x}^2+\sum_{N\gg 1} N^{2s} |P_{\ll N} f |_{L^\infty_x}  |\tilde{P}_N g|_{L^2_x}^2\\
 & \lesssim |f|_{L^\infty_x}^2 |g|_{H^s}^2
 \end{align*}
To estimate the  other terms in the right-hand side of \eqref{vv} we separate the cases $ s<0 $ and $s\ge 0 $. For $ s\ge 0 $, 
we notice that 
$$
|P_{\lesssim 1} (P_{\gg1} f \, g) |_{L^2_x} \lesssim |f|_{L^\infty_x} |g|_{L^2_x}\lesssim |f|_{L^\infty_x} |g|_{H^s}
$$
and  that  Bernstein inequalities lead to
\begin{align*}
 \sum_{N\gg 1} N^{2s} |P_N(P_{\gtrsim N} f g)|_{L^2_x}^2 & 
 \lesssim |g|_{L^2_x}^2  \sum_{N\gg 1} N^{0-} |P_{\gtrsim N} f|_{W^{s+,\infty}}^2\lesssim |g|_{H^s}^2  | f|_{W^{s+,\infty}}^2 \;.
\end{align*}
This completes the proof of \eqref{prod5} for $ s\ge 0$.
Now for $ s< 0 $ we use  Bernstein inequalities   to get 
\begin{align*}
|P_{\lesssim 1} (P_{\gg1} f\, g) |_{L^2_x}  \lesssim \sum_{K\gg 1} |P_{K} f P_{\lesssim K} g|_{L^2_x} 
&\lesssim \sum_{K\gg 1} |P_{K}  f|_{L^\infty_x} |P_{\lesssim K}g|_{L^2_x}\\
& \lesssim \sum_{K\gg 1} K^{-(|s|+)} |P_{K}  f|_{W^{|s|+,\infty}}  K^{|s|} |P_{\lesssim K}g|_{H^s}\\
& \lesssim  | f|_{W^{|s|+,\infty}}   |g|_{H^s}
\end{align*}
and in the same way
\begin{align*}
 \sum_{N\gg 1} N^{2s} |P_N(P_{\gtrsim N} f g)|_{L^2_x}^2 & 
 \lesssim 
  \sum_{N\gg 1} N^{2s} \sum_{K\gtrsim N} \Bigl( K^{-(|s|+)}|P_K f|_{W^{|s|+,\infty}}  K^{|s|}|\tilde{P}_{\lesssim K} g|_{H^s_x}\Bigr)^2\\
 & \lesssim  | f|_{W^{|s|+,\infty}}   |g|_{H^s}\; .
\end{align*}
This completes the case $ s<0 $.

It remains to prove \eqref{prod6}. We separate the  case 
 $ N\lesssim 1 $ and $ N \gg 1$. For $ N\lesssim 1 $ we have 
 $$
 N^\theta |[P_N, P_{\ll N} f] g_x |_{L^2_x} \lesssim | P_{\ll N} f|_{L^\infty_x}  |P_{\lesssim 1}  g_x |_{L^2_x} \lesssim |f|_{L^\infty_x} |g|_{H^\theta}  \; .
 $$
 
  For $ N\gg 1$, we separate the contributions of $ P_{\ll N} f $ and $ P_{\gtrsim N} f $.

  The contribution of the first term is easily estimated thanks to \eqref{cm1} by 
 $$
 N^\theta |[P_N, P_{\ll N} f] g_x |_{L^2_x} \lesssim N^\theta  |f|_{L^\infty_x} |\tilde{P}_N g|_{L^2_x}\lesssim \delta_N  |f|_{L^\infty_x} |g|_{H^\theta} \; .
 $$
 For the second term we  start by writing 
 \begin{align*}
 N^\theta |[P_N, P_{\gtrsim N} f] g_x |_{L^2_x}&\lesssim  N^\theta |P_{\gtrsim N} f P_N g_x|_{L^2_x} +N^\theta |P_N( P_{\gtrsim N} f g_x)|_{L^2} \\
 \end{align*}
 Then Bernstein inequalities lead to 
 $$
 N^\theta |P_{\gtrsim N} f P_N g_x|_{L^2_x} \lesssim N^\theta   N   |P_{\gtrsim N} f |_{L^\infty_x} |P_N g|_{H^\theta}\lesssim 
 \delta_N |f_x|_{L^\infty_x} | g|_{H^\theta}
 $$
 and proceeding as above we get for $\theta\ge 0 $,
 \begin{align*}
 \sum_{N\gg 1} N^{2\theta} |P_N(P_{\gtrsim N} f g_x)|_{L^2_x}^2 & = \sum_{N\gg 1} N^{2\theta} |P_N(P_{\gtrsim N} f P_{\lesssim 4 N} g_x)|_{L^2_x}^2\\
&  \lesssim |g|_{L^2_x}^2  \sum_{N\gg 1} N^{0-} N^2 |P_{\gtrsim N} f|_{W^{\theta+,\infty}}^2\lesssim |g|_{H^\theta}^2  | f_x|_{W^{\theta+,\infty}}^2 \;.
\end{align*}
and for $\theta<0 $
\begin{align*}
 \sum_{N\gg 1} N^{2\theta} |P_N(P_{\gtrsim N} f g_x)|_{L^2_x}^2 & 
 \lesssim 
  \sum_{N\gg 1} N^{2\theta} \sum_{K\gtrsim N} \Bigl( K^{-(|\theta|+)}|P_K f|_{W^{|\theta|+,\infty}}  K K^{|\theta|}|\tilde{P}_{\lesssim K} g|_{H^\theta}\Bigr)^2\\
 & \lesssim  | f_x|_{W^{|\theta|+,\infty}}^2   |g|_{H^\theta}^2\; .
\end{align*}

  \vspace*{2mm}  

\subsection{Ellipiticity and estimates  for the operator   $\mfT_b:=h_b(1+\mu \mathcal{T}[h_b])$}

In the following we give two lemmas concerning the operator $\mfT_b:=h_b (1+\mu\mathcal{T}[h_b])$, and its inverse. The first lemma is a straightforward adaptation of Lemma 1 in \cite{Israwi11}.
We give the proof for sake of completness. 
\begin{lemma}\label{proprim}
Let be given $\mu>0 $ and $b\in C_b^{1}(\R)$ satisfying  \eqref{hb}. Then the  self-adjoint  operator
 $$
\mfT_b: H^1(\R)\longrightarrow H^{-1}(\R)
$$
is well defined, one-to-one and onto.
\end{lemma}
\begin{proof} We rewrite $\mfT_b u$ as 
$$
\mfT_b u= h_b u -\frac{\mu}{3} (h_b^3 u_x)_x + \frac{\beta \mu}{2}
\Bigl( (h_b^2 b_x u)_x -h_b^2 b_x u_x\Bigr) + \beta^2 \mu h_b b_x^2 u
$$
We notice that for  $b\in C^1_b(\R) $,  $\mfT_b u$ is clearly continuous from $ H^1(\R) $ into $H^{-1}(\R) $ and we  define the  continuous bilinear form $ a(\cdot,\cdot)$ on $ H^1$ by 
$$
a(u,v) = (\mfT_b u, v)_{H^{-1},H^1} \; .
$$
Observe that, the writing \eqref{selfadjoint} of  $\mfT_b u$ ensures that $ a(\cdot,\cdot)$ is symmetric as soon as $ b\in C^2_b(\R) $.
The crucial algebraic property is that $a(u,u)$ may be rewritten as 
$$
a(u,u)=\int_{\R} h_b u^2 + \mu \int_{\R} h_b \Bigl( \frac{h_b}{\sqrt{3}} u_x -\frac{\sqrt{3}\beta}{2} b_x u\Bigr)^2 +\frac{\mu \beta^2 }{4}\int_{\R} h_b b_x^2 u^2
$$
so that 
\begin{equation}\label{coer}
a(u,u) \ge h_0\Bigl(  |u|_{L^2_x}^2 + \mu  \Bigl|\frac{h_b}{\sqrt{3}} u_x -\frac{\sqrt{3}\beta}{2} b_x u \Bigr|_{L^2_x}^2+ \frac{\mu  \beta^2}{4}
|b_x u|_{L^2_x}^2\Bigr) \; . 
\end{equation}
On the other hand,  observe that 
\begin{align*}
|u|_{H^1}^2 &\le |u|_{L^2_x}^2 + \frac{3}{h_0^2}|\frac{h_b}{\sqrt{3}} u_x|_{L^2_x}^2 \\
    & \le |u|_{L^2_x}^2+ \frac{6}{h_0^2}\Bigl( \Bigl|\frac{h_b}{\sqrt{3}} u_x -\frac{\sqrt{3}\beta}{2} b_x u \Bigr|_{L^2_x}^2
    + \frac{3\beta^2}{4} |b_x u |_{L^2_x}^2 \Bigr) \; .
\end{align*}
so that 
$$
 \max(1, \frac{18}{\mu h_0^2} ) \Bigl( |u|_{L^2_x}^2 + \mu  \Bigl|\frac{h_b}{\sqrt{3}} u_x -\frac{\sqrt{3}\beta}{2} b_x u \Bigr|_{L^2_x}^2+ \frac{\mu  \beta^2}{4}
|b_x u|_{L^2_x}^2\Bigr) \ge |u|_{H^1}^2 \; .
$$
Combining this last inequality with \eqref{coer}, we get  
$$
a(u,u) \ge \frac{h_0}{\max(1, \frac{18}{\mu h_0^2} )}
|u|_{H^1}^2 \; .
$$
This proves  that $ a(\cdot,\cdot) $ is coercive on $ H^1(\R) $ and Lax-Milgram theorem then ensures that for any $ f\in H^{-1}(\R) $, there exists a unique $ u\in H^1(\R) $ such that $ \mfT_b u =f $, i.e. $\mfT_b u  $ is a bijection  from $H^1(\R) $ to $H^{-1}(\R)$.
\end{proof}
The following lemma gives more properties of the inverse operator $\mfT^{-1}_b$.
\begin{lemma}\label{proprim'}
Let $s>-\frac{3}{2}$ and $b\in   W^{(s+2)\vee \frac{5}{2}+,\infty}(\R)$  such that \eqref{hb} is satisfied. Then 
\begin{enumerate}
\item  For any $ s\ge 0 $ and $f\in H^s(\R) $ it holds 
\begin{equation}\label{estT}
\vert \mfT^{-1}_bf\vert_{H^s}+\sqrt{\mu}\vert \partial_x\mfT^{-1}_bf\vert_{H^s}\leq
C \vert f\vert_{H^s} \; .
\end{equation}
where $ C= C(\frac{1}{h_0},|b|_{W^{(s+2)\vee \frac{5}{2}+,\infty}} )$. \\
\item For $-\frac{3}{2}< s< 0$  there exists 
$ C_0=C_0(\frac{1}{h_0},|b|_{W^{\frac{5}{2}+,\infty}}) >0 $ such that   if $ 0<\beta <C_0 $ then \eqref{estT} holds actually for any $ f\in H^{-1}(\R) $.
\end{enumerate}
\end{lemma}
\begin{proof}
The proof in the case $ s\ge 0 $ is given in  \cite{Israwi11}. 
According to Lemma \ref{proprim}, $ \mfT^{-1}_b $ is well defined from $H^{-1}(\R) $ into $ H^1(\R) $. So let $ f\in H^{-1}(\R) $ and set $ u= \mfT^{-1}_b f \in H^1(\R) $. According to \eqref{selfPN}, for any $ N\ge 1 $ it holds 
$$
\mfT_b( P_N u) = P_N f -\frac{\mu}{3} \partial_x \Bigl( [P_N, h_b^3 \partial_x] u\Bigr) -[P_N, g_\beta] u = y_N + \sqrt{\mu} \partial_x z_N
$$
with $ y_N = P_N f - [P_N, g_\beta] u \in L^2(\R) $ and $z_N=\frac{\sqrt{\mu}}{3}[P_N, h_b^3 \partial_x] u \in L^2(\R)$. Then according to (\cite{Israwi11}, proof of Lemma 2, Step 1.) we have 
\begin{equation}\label{o1}
\|P_N u\|_{H^1_\mu} \le C(\frac{1}{h_0}) \Bigl( \| y_N\|_{L^2} +  \| z_N\|_{L^2} \Bigr)\; .
\end{equation}
We notice that $ [P_N,g_\beta]=\beta [P_N, \tilde{g}_\beta]$ with $ \tilde{g}_\beta= b +\frac{\mu}{2} \partial_x(h_b^2 b_x) +\mu \beta b_x^2 $. Therefore making use of \eqref{prod5}  we obtain for $ -3/2< s<0 $ 
\begin{align}
N^s  |P_N f - [P_N,g_\beta] u|_{L^2}  \lesssim & \delta_N \Bigl( |P_N f|_{H^s} + \beta  | b |_{W^{\frac{3}{2}+,\infty}} |u|_{H^{s}}\nonumber \\
& + \beta(|(h_b b_x)_x|_{W^{\frac{1}{2}+,\infty}}+|b_x|_{W^{\frac{1}{2}+,\infty}})|u|_{H^{s+1}_\mu} \Bigr)\nonumber \\
 \lesssim & \delta_N \Bigl( |P_N f|_{H^s} + C(|b|_{W^{\frac{5}{2}+,\infty}}) \,  \beta 
 |u|_{H^{s+1}_\mu} \Bigr) \label{o2}\; .
\end{align}
where we used the factor $ \mu $ and the fact that $s+1>-1/2 $.
On the other hand, making use of \eqref{prod6} we get 
\begin{eqnarray}
N^s |[P_N, h_b^3 \partial_x ] u|_{L^2} & \lesssim &|\partial_x (h_b^3)|_{L^\infty} |\tilde{P}_N u |_{H^s} + N^{0-} | \partial_x (h_b^3) |_{W^{\frac{3}{2}+,\infty}} |u|_{H^{-\frac{3}{2}+}}  \nonumber \\
&\lesssim & \beta |b_x|_{W^{\frac{3}{2}+,\infty}}\Bigl( |\tilde{P}_N u |_{H^s} +  N^{0-} |u|_{H^{-\frac{3}{2}+}} \Bigr)\label{o3}
\end{eqnarray}
Gathering \eqref{o1}-\eqref{o3}, we eventually get for $ s<0 $, 
\begin{eqnarray}
|u|_{H^{s+1}_\mu}^2 & \le &   \sum_{N\ge 1} N^{2s} |P_{N} u|_{H^1_\mu}^2\nonumber \\
& \lesssim & |f|_{H^s}^2 + C(|b|_{W^{\frac{5}{2}+,\infty}})\, \beta^2 |u|_{H^{s+1}_\mu}^2 \; .
\end{eqnarray}
 where the implicit constant only depends on $\frac{1}{h_0} $. Therefore there exists $ C_0= C_0(\frac{1}{h_0}, |b|_{W^{\frac{5}{2}+,\infty}}) $ such that  for any $ 0<\beta<C_0 $ and any $ -\frac{3}{2}\le s<0 $, \eqref{estT} holds.
 \end{proof}
\subsection{Proof of the positivity of $ h=1+\eps\zeta-\beta b$}
 We prove the positivity of $ h=1+\eps\zeta -\beta b$ as soon as $ \zeta_0\in H^s(\R) $,  with $s>1/2$, 
\begin{lemma}\label{lemminmax}
Let  $s>1/2$ and let $\zeta_0 \in H^s(\R)$ and $ \zeta \in C([0,T],H^s(\R)  )$ satisfying the first equation of (\ref{BPW}) with  $ b\in H^{s+1} $ and $ u\in C([0,T] ; H^{s+1}(\R)). $
 If $1+\eps\zeta_0-\beta b>0$ on $ \R $ then $\inf_{[0,T]\times \R} 1+\eps\zeta-\beta b>0 $.
\end{lemma}
\begin{proof} We start by assuming that $\zeta_0\in H^{s+1}(\R)  $ so that $ \zeta\in C^1([0,T]\times \R) $. Since $u\in C([0,T] ; H^{s+1}(\R))\hookrightarrow C([0,T]; C^1(\R)) $ we can introduce the flow $ q(\cdot,\cdot) $ associated with $ u $ that is defined by 
 \begin{equation}\label{defq}
  \left\{ 
  \begin{array}{rcll}
  q_t (t,x) & = & \eps u(t,q(t,x))\, &, \; (t,x)\in [0,T] \times \R\\
  q(0,x) & =& x\, & , \; x\in\R \; 
  \end{array}
  \right. .
  \end{equation} 
Then the first equation in \eqref{4} ensures that 
$$
\frac{d}{dt} h(t,q(t,x))= -\eps u_x(t,q(t,x)) h(t,q(t,x))
$$
which leads to 
$$
h(t,q(t,x))=h(0,x) \exp \Bigl( -\eps\int_0^t u_x(s,q(s,x))\, ds\Bigr) ,\quad \forall (t,x)\in [0,T]\times \R\; .
$$
In particular, since for any $t\in [0,T] $, $q(t,\cdot) $ is an increasing $ C^1$-diffeomorphism of $ \R $ we deduce that 
\begin{equation}\label{fin3}
\inf_{[0,T]\times \R} |h|\ge e^{-\eps T \|u_x\|_{L^\infty(]0,T[\times \R)}} \inf_{\R} h_0 \; 
\end{equation}
which proves the result at this regularity. For $ \zeta_0\in H^s(\R) $, with $ s>1/2 $, it suffices to approximate $\zeta_0 $ by a sequence of smooth initial data $ (\zeta_{0,n})_{n\ge 0}  $ and pass to the limit on the associated sequence of solutions $(\zeta_n)_{n\ge 0} $ to the first equation of \eqref{BPW}. The same commutator estimates  that we used in Section \ref{section3} enable to prove that  the sequence of smooth solutions $(\zeta_n)_{n\ge 0} $ converges toward $ \zeta $ in $C([0,T]; H^s(\R))\hookrightarrow C([0,T]\times \R) $ and thus $ h=1+\eps\zeta -\beta b$ also satisfies \eqref{fin3}.
\end{proof}

\providecommand{\href}[2]{#2}


\begin{thebibliography}{10}
\bibitem{Ada} K. Adamy, Existence of solutions for a Boussinesq system on the half line and on finite interval,  {\it D.C.D.S A} {\bf 29} (2011), 25-49.
\bibitem{Ami}
C.~J.~Amick, Regularity and Uniqueness of Solutions for the Boussinesq System of Equations, {\it Journal of Differential Equations }{\bf 54} (1984), 231-247. 
, no.1, 49-96.

\bibitem{BBM}
T.~B.~Benjamin, J.~L.~Bona and J.~J.~Mahoney, Model equations for long 
waves in nonlinear dispersive systems, {\it Phil. Trans. Roy. Soc. London A} 
{\bf 227} (1972), 47--78.

\bibitem{BCS02}
J.~L.~Bona, M.~Chen and J.~C.~Saut, Boussinesq equations and other systems for small-amplitude long waves
in nonlinear dispersive media I: Derivation and the linear theory, {\it J. Nonlinear Sci.} {\bf 12} (2002), 283-318. 

\bibitem{BCS04}
J.~L.~Bona, M.~Chen and J.~C.~Saut, Boussinesq equations and other systems for small-amplitude long waves
in nonlinear dispersive media II: The nonlinear theory, {\it J. Nonlinearity } {\bf 17} (2004), 925-952. 

\bibitem{Bou1} 
J. ~Boussinesq. Th\'eorie g\'en\'erale des mouvements qui sont propag\'es dans un canal rectangulaire horizontal.  {\it Comptes Rendus des S\'eances de l’Acad\'emie des Sciences. Paris.} (1871),  256–260.
\bibitem{Bou2} 
J. ~Boussinesq. Théorie des ondes et des remous qui se propagent le long d’un canal rectangulaire horizontal, en communiquant au liquide contenu dans ce canal des vitesses sensiblement pareilles de la surface au fond.  {\it Journal de Mathématiques Pures et Appliquées. Paris.} (1872),  55–108.



\bibitem{DunSchw}
Dunford-Schwartz, Linear Operators. PartI. Pure and Applied Mathematics. Vol VII, {\it Interscience, New York}. 

\bibitem{DIs}
V. ~Duchêne,  S. ~Israwi,  Well-posedness of the Green-Naghdi and Boussinesq-Peregrine systems, {\it Annales mathématiques Blaise Pascal}, (2018)



\bibitem{Lanlivre13}
D.~Lannes. {\it The water waves problem: mathematical analysis and asymptotics.} Mathematical Surveys and Monographs (AMS) (2013).


\bibitem{Lan06}
D.~Lannes.  Sharp estimates for pseudo-differential operators with symbols of limited smoothness and commutators,  {\it J. Funct. Anal. } {\bf 232} (2006), 495-539.


\bibitem{Israwi11} 
 S.~ Israwi.  Large time existence for 1D Green-Naghdi equations.
 {\it Nonlinear Anal.,} {\bf 74} (2011), 81--93.

\bibitem{Mes}
B. ~Mésognon-Gireau. The Cauchy problem on large time for a Boussinesq-Peregrine equation with large topography variations. {\it Adv. Differential Equations}, {\bf 22(7/8)}(2017), 457–504.

\bibitem{MTZ} 
L.~Molinet, R.~Talhouk and I.~Zaiter.  The  Boussinesq Systems revisited.  {\it Nonlinearity} {\bf 34,} (2021), 744-775.

\bibitem{Per} 
D. ~H.~ Peregrine.  Long waves on a beach.  {\it Journal of Fluid Mechanics,}  (1967), 815–827,.

\bibitem{SX}
J.~C.~Saut and L.~Xu, The Cauchy Problem on Large Time for Surface-Waves-Type Boussinesq Systems,  {\it J. Math. Pures Appl. } {\bf 97} (2012), 635-662. 

\bibitem{SWX}
J.~C.~Saut, C.~Wang and L.~Xu, The Cauchy Problem on Large Time for Surface-Waves-Type Boussinesq Systems II,  {\it SIAM Math. Anal. } {\bf 49 (4)} (2017), 2321-2386. 

\bibitem{Sch}
M.~E.~Schonbek, Existence of Solutions to the Boussinesq System of Equations, {\it Journal of Differential Equations }{\bf 42} (1981), 325-352. 

\bibitem{S87}
J.~Simon,  compact sets in the space L$^p(0,T;\mbox{b})$, {\it  Ann. Mat. Pura appl. }{\bf 146} (1987),65-96.

\bibitem{Taylor91}
M.~I.~Taylor, {\it Tools for PDE, Pseudodifferential Operators, Paradifferential Operators, and Layer Potentials.}  Mathematical Surveys and Monographs vol. 81(AMS) (1991).

\end{thebibliography}
\end{document}